\documentclass{article}
\usepackage[a4paper,top=3cm,bottom=3cm,left=3cm,right=3cm,marginparwidth=1.75cm]{geometry}

\usepackage{graphicx} 
\usepackage{epstopdf}
\usepackage[normalem]{ulem}
\usepackage{amsmath}
\usepackage{bm}
\usepackage{amsmath, amsfonts, amssymb, amsthm}
\usepackage{mathrsfs, dsfont}
\usepackage[dvipsnames]{xcolor}
\usepackage{graphicx}
\usepackage{float}
\usepackage{comment}
\usepackage{subfig}
\usepackage{bigints}

\usepackage{xparse} 

\usepackage{tasks}
\settasks{style=itemize}

\usepackage[numbers,square]{natbib}
\usepackage[colorlinks=true, allcolors=blue]{hyperref}

\usepackage{fancyhdr}
\newtheorem{theorem}{Theorem}[section]
\newtheorem{lemma}[theorem]{Lemma}

\newtheorem{proposition}[theorem]{Proposition}

\theoremstyle{remark}
\newtheorem{remark}{Remark}[section]
\newtheorem{example}[remark]{Example}

\numberwithin{equation}{section}

\usepackage{pifont}

\let\oldremark\remark
\let\oldendremark\endremark

\renewenvironment{remark}
  {\begingroup
   \pushQED{\qed}%
   \oldremark}
  {\popQED
   \oldendremark
   \endgroup}

\let\oldexample\example
\let\oldendexample\endexample

\renewenvironment{example}
  {\begingroup
   \pushQED{\qed}%
   \oldexample}
  {\popQED
   \oldendexample
   \endgroup}

\def\A{\mathbb{A}}

\def\B{\mathbb{B}}
\def\S{\mathbb{S}}

\def\P{\mathbb{P}}
\def\Pc{\mathcal{P}}
\def\F{\mathcal{F}}

\def\T{\mathcal{T}}
\def\H{\mathbb{H}}

\def\E{\mathbb{E}}
\def\Fb{\mathbb{F}}

\def\Es{\mathscr{E}}

\def\Ac{\mathcal{A}}

\def\Zc{\mathcal{Z}}
\def\Dc{\mathcal{D}}
\def\loc{\text{loc}}
\newcommand{\rd}{\mathrm{d}}

\newcommand{\esssup}{\operatorname*{ess\,sup}}

\newcommand{\tr}[1]{\operatorname{Tr}\brac{#1}}

\def\cara{\text{exp}}
\def\b{{\bl b}}

\def\x{\mathbf{x}}
\def\conc{\text{conc}}

\def\N{\Tilde{N}}

\newcommand{\R}{\mathbb{R}}
\renewcommand{\L}{\mathbb{L}}

\renewcommand{\N}{\mathbb{N}}
\newcommand{\C}{\mathcal{C}}
\newcommand{\I}{\mathbb{I}}
\newcommand{\cI}{\mathcal{I}}
\newcommand{\Rc}{\mathcal{R}}
\newcommand{\Ec}{\mathcal{E}}
\newcommand{\M}{\mathcal{M}}

\newcommand{\ch}[1]{\mathds{1}_{\left\{ #1 \right\}}}

\DeclareMathOperator*{\argmax}{arg\,max}

\ExplSyntaxOn
\NewDocumentCommand{\brac}{O{1} m}
 {
  \str_case:nnF {#1}
   {
    {1}{(#2)}
    {2}{\bigl(#2\bigr)}
    {3}{\Bigl(#2\Bigr)}
    {4}{\biggl(#2\biggr)}
    {5}{\Biggl(#2\Biggr)}
   }
   {(#2)}
 }

\NewDocumentCommand{\mbrac}{O{1} m}
 {
  \str_case:nnF {#1}
   {
    {1}{[#2]}
    {2}{\bigl[#2\bigr]}
    {3}{\Bigl[#2\Bigr]}
    {4}{\biggl[#2\biggr]}
    {5}{\Biggl[#2\Biggr]}
   }
   {(#2)}
 }

\NewDocumentCommand{\set}{O{1} m}
 {
  \str_case:nnF {#1}
   {
    {1}{\{#2\}}
    {2}{\bigl\{#2\bigr\}}
    {3}{\Bigl\{#2\Bigr\}}
    {4}{\biggl\{#2\biggr\}}
    {5}{\Biggl\{#2\Biggr\}}
   }
   {(#2)}
 }

 \NewDocumentCommand{\abs}{O{1} m}
 {
  \str_case:nnF {#1}
   {
    {1}{|#2|}
    {2}{\bigl|#2\bigr|}
    {3}{\Bigl|#2\Bigr|}
    {4}{\biggl|#2\biggr|}
    {5}{\Biggl|#2\Biggr|}
   }
   {(#2)}
 }
\ExplSyntaxOff

\usepackage{mathtools}
\usepackage{shuffle}
\usepackage{bbm}
\usepackage{authblk}

\mathtoolsset{showonlyrefs}
\newcommand{\bl}{\color{blue}}

\usepackage{enumitem}

\title{Backward SDE characterization of the finite horizon
Principal-Agent problem
\footnote{This work is partially supported by NSF grant \#DMS-2508581 and the Chair {\it Finance and Sustainable growth} of the 
{\it Risk Foundation}}} 
\author{
Yuxing Huang\thanks{New York University, Tandon School of Engineering. Email: \texttt{yh6004@nyu.edu}}~~~~~~
Nizar Touzi\thanks{New York University, Tandon School of Engineering. Email: \texttt{nizar.touzi@nyu.edu}}
}

\begin{document}

\maketitle

\begin{abstract}
We consider the finite horizon continuous-time Principal--Agent problem under deterministic discount factors. Following the Sannikov reduction to a stochastic control problem, we provide a further characterization of the Principal's value function in terms of a backward SDE inducing the corresponding optimal contract. In particular, this allows us to bypass the fully nonlinear HJB equation satisfied by the Principal's value function in the Markovian setting. This new approach allows to handle a new class of Principal-Agent problems which was not accessible with the existing method, namely the setting where the Agent faces a regime-switching control problem.
\end{abstract}




\section{Introduction}\label{sec:introduction}

The continuous-time Principal–Agent problem is a canonical framework used to capture moral‐hazard situations, in which a principal delegates a task to an agent whose effort is not directly observable. Mathematically, this is formulated as a Stackelberg game, also called bi-level optimization in the operations research literature, or leader-follower problem in optimal control theory, and defined as follows:
\begin{itemize}
    \item given a contract $\xi$ representing the lump sum random compensation at maturity, the agent commits to the contract if his utility value is above his participation level and, in this case, chooses his best effort as an optimal response to the announced compensation scheme;
    \item the principal optimally chooses the contract $\xi$, anticipating the agent's best response, and accounting for the agent's participation constraint.
\end{itemize} 
Originating with \citet*{RePEc:aea:aecrev:v:63:y:1973:i:2:p:134-39} and formalized in continuous time by \citet*{Holmstrom1987}, the model has now become a central tool in modern corporate finance and contract theory. A fundamental development in the subject is Sannikov's \cite{sannikov2004dynamic} reduction method, which was later related by \citet*{cvitanic2017dynamicprogrammingapproachprincipalagent} to the backward stochastic differential equation (SDE) approach to path dependent optimal control. Sannikov's key insight was to describe incentive-compatible contracts through the agent's continuation value process, which leads to the following representation 
$$
Y_T=\xi,
~~\text{and}~~
dY_t
=
Z_t\cdot dX_t
-
H^A_t(X_{\cdot \wedge t},Y_t,Z_t)dt,
~~t\in[0,T],
 $$
where $H^A$ is the agent's Hamiltonian, and the  process $Y$ represents the agent's continuation value. This representation provides a tractable parametrization of
contracts and reduces the original Stackelberg game to a
standard stochastic control problem. This method was later rigorously formulated in a probabilistic setting by \citet*{cvitanic2017dynamicprogrammingapproachprincipalagent}, who studied a general case where the agent also controls the volatility of the output process by using the theory of second-order BSDEs of \citet*{Soner2012} and further developed by \citet*{PossamaiTanZhou2018}. The unbounded horizon setting of Sannikov is addressed in \citet*{nizar_random_horizon_principal_agent}.
\\

This reduction has since generated a substantial literature, including sequential contracting by \citet*{AlvarezBayraktarEkrenHuang2025}, multi-agent and mean-field contracting problems by \citet*{EliePossamai2019}, \citet*{ElieMastroliaPossamai2019}, \citet*{CarmonaWang2021FiniteStateContractTheory}, and \citet*{HernandezSantibanez2024PrincipalMultiagents}, multi-principal problems by \citet*{hu2022principalagentproblemmultipleprincipals}, contractual frictions involving one-sided commitment and endogenous quitting by \citet*{ZhangZhu2025OneSidedCommitment}, adverse selection under a single contract restriction by \citet*{alvarez2026principalagentproblemsadverseselection}, and an alternative approach to volatility control that avoids the use of second-order BSDEs by ~\citet*{chiusolo2025newapproachprincipalagentproblems}.  Beyond these theoretical developments, continuous-time contracting models have also found applications in delegated portfolio management and brokerage by
\citet*{NadtochiyZariphopoulou2019}, 
\citet*{AlvarezNadtochiyWebster2023}, and
\citet*{AlvarezNadtochiy2026}, market regulation by
\citet*{Euch_market_regulation}, cybersecurity by
\citet*{mastrolia2026agencyproblemsadversarialbilevel}, and demand-response programs in electricity retail markets by \citet*{AidPossamaiTouzi2022DemandResponse},
\citet*{ElieHubertMastroliaPossamai2021}, and
\citet*{HernandezJofrePossamai2023}. For related computational approaches, see \citet*{CampbellChenShrivatsJaimungal2021}, \citet*{DayanikliLauriere2025MachineLearningStackelbergMFG}, and \citet*{LudkovskiXieZhu2025}.\\

Nevertheless, Sannikov's reduction does not provide a complete solution to the principal's problem, but rather expresses it as a new control problem with an additional controlled state, thus identifying optimal contracts through the optimal controls of the reduced problem. In the Markovian setting, the principal's value is typically characterized by the fully nonlinear Hamilton--Jacobi--Bellman (HJB) equation on the enlarged state space consisting of the output process and the agent's continuation utility: \begin{align*}
    -\partial_t v_P(t,x,y)
-\tfrac12 \Delta_{xx}v_P(t,x,y)-
H_P(t,x,Dv_P(t,x,y),D^2v_P(t,x,y))
=0,
\end{align*}
with an appropriate terminal condition. Here $H_P$ is the principal's Hamiltonian. Since the principal controls the process $Z$, which enters the diffusion coefficient of the continuation-value process, the reduced control problem involves an unbounded control in the diffusion coefficient. The corresponding HJB equation is therefore a fully nonlinear, degenerate elliptic PDE. Typically, in order to identify an optimal contract from this equation, one needs a verification theorem which requires sufficient smoothness of the value function $v_P$. 
In the non-Markovian setting, one may use instead the path-dependent PDEs as in \citet*{ekren2014viscosity}. Identifying the optimal controls faces the same regularity issue, which is even more serious in this setting.\\

The main contribution of this paper is to provide an explicit backward SDE representation for the principal’s value in the non-Markovian setting, and to identify the optimal contract under certain conditions. We study a finite-horizon moral-hazard problem in which the agent controls the drift of the output, while the principal offers a contract consisting of a terminal lump-sum compensation $\xi$ and a continuous contractible policy $\beta$, such as a running payment or operational support. A crucial structural assumption is that the discount rate is deterministic and uncontrolled. The representation is obtained through an auxiliary  pair process $(P,Q)$ characterized as the unique solution of an auxiliary backward SDE. If $r$ denotes the agent's discount rate and $R$ is the agent's reservation utility,  then we prove that the principal's value can be expressed in terms of the concave envelope $g^{\conc}$ of  the principal's utility function $g$:
$$
    V^P
    =
    g^{\conc}
    \Big((P_0-R)e^{\int_0^T r(s)ds}\Big),
$$
and the optimal agent's effort is expressed explicitly in terms of the process $Q$. Thus the stochastic and non-Markovian features of the problem are encoded in the auxiliary pair process $(P,Q)$, while the principal's preference enters through its concave envelope. Moreover, we construct an optimal contract explicitly when $(P_0-R)e^{\int_0^T r(s)ds}$ lies in the contact set between $g$ and its concave envelope $g^{\conc}$. In the Markovian case, this representation reduces the usual fully nonlinear HJB equation to a semilinear PDE for the function associated with $P$.
A distinctive feature of our setting is that $g$ is only assumed to be non-decreasing and need not be concave. Economically, this relaxation allows the model to capture threshold effects and benchmark-type objectives. Our analysis shows that the concavity of the principal's value is a property of the dynamic contracting problem itself rather than a consequence of the utility function. \\

The same method also applies beyond the risk-neutral-agent benchmark. We
consider, in particular, a risk-averse formulation in which both the agent and the principal have exponential utilities. Although continuous-time principal--agent models with exponential utilities have been studied extensively
since \citet*{Holmstrom1987} and \citet*{SCHATTLER1993331}, and further developed by \citet*{Sung1995}, \citet*{SchattlerSung1997}, \citet*{Muller1998} and \citet*{HellwigSchmidt2002}, the reduction we obtained here is new. After an exponential transformation of the continuation values of the agent and the principal, the non-Markovian second-best problem is reduced to a single auxiliary BSDE. In contrast with the risk-neutral-agent case, this auxiliary equation is no longer Lipschitz, and it is a quadratic BSDE whose generator combines the cost of incentive provision with the optimal risk-sharing term; see, for example, the work by \citet*{kobylanski_quadra_bsde}, \citet*{BRIAND20132921} and \citet*{Tevzadze2008}. Solving this BSDE directly yields the principal's value and an optimal contract. Related risk-averse and exponential-utility formulations also appear in the work by \citet*{CvitanicWanZhang2009}, \citet*{Williams2015} and \citet*{Euch_market_regulation}.\\

We finally show that this new reduction applies to a delegation problem in which the agent controls not only hidden effort but also unobservable costly switches among finitely many operating regimes. Such models are natural when the agent chooses when to change production modes, machine configurations, logistics regimes, or generation technologies. For a fixed contract, the agent's continuation values are described by a system of obliquely reflected BSDEs, following the probabilistic approach to optimal switching developed by \citet*{hamadene_on_the_starting_and_stopping_problem}, \citet*{hu2007multidimensionalbsdeobliquereflection} and \citet*{Hamadne2010SwitchingPA}. The Sannikov method unfortunately does not apply in this context as one cannot simply turn the backward dynamics into a forward controlled additional state due to the following major difficulties. First, the principal offers a single terminal transfer, so the corresponding reflected BSDE system
must share a common terminal condition across regimes. Second, the finite-variation processes introduced in the system of reflected BSDEs are determined by the switching obstacles and the associated Skorokhod conditions, and therefore cannot be chosen by the principal as additional
control variables. Hence, the usual Sannikov-type reversal of the agent's continuation dynamics does not lead to a standard stochastic control problem. Nevertheless, the auxiliary-BSDE reduction developed in this paper still works. We construct an auxiliary system of reflected BSDEs which plays the role of the former pair $(P,Q)$ in the non-switching case. Following the same scheme as in the non-switching case, we can characterize the principal's value and obtain an explicit optimal contract under the corresponding attainability condition. This extension shows that our method is not tied to the reduction to a standard stochastic control problem.\\

The remainder of the paper is organized as follows. Section~\ref{sec:PA_problem} introduces the baseline model and recalls the continuation-value formulation of admissible contracts. Section~\ref{sec:BSDE_chara} provides the concavification result and the BSDE representation of both the principal's value and an optimal contract. Section~\ref{sec:cara} treats the case of exponential utilities and derives the corresponding quadratic-BSDE characterization. The final section~\ref{sec:PA_RS} studies the regime-switching agent and provides the reflected-BSDE representation of the principal's value and an optimal contract.

\paragraph{Notations.}
Fix $d\in\N$, let $\Omega = C([0,T],\R^d)$ be the set of all continuous maps from $[0,T]$ to $\R^d$. Let $X:[0,T]\times \Omega\longrightarrow\R^d$ be the canonical process on $\Omega$, representing the output that the agent is responsible for, $i.e.$ $X_t(\omega) := \omega_t$. Let $\mathbb{F}^0 = \brac{\F_t^0}_{t\in[0,T]}$ denote the canonical filtration, i.e., $\F_t^0:=\sigma(X_s:s\leq t)$. We equip $(\Omega,\F_T)$ with the Wiener measure $\P_0$, under which $X$ is a Brownian motion. We let $\mathbb{F} = (\F_t)_{ t\in[0, T]}$ be the usual augmentation of $\mathbb{F}^0$, that is, each $\F_t$ is first completed with the $\P_0$-null sets and then replaced by its right-continuous modification. We use $\T_t$ to denote the set of $\Fb$-stopping times taking values in $[t,T]$.\\

Let $\mathfrak{P}$ be a set of probability measures on $(\Omega,\F_T)$:

\begin{itemize}\renewcommand\labelitemi{-}

    \item Let $\S^p(\mathfrak{P})$ denote all the càdlàg processes $\eta: [0,T]\times\Omega\longrightarrow \R^k$, s.t. $\sup_{\P\in\mathfrak{P}}\E^{\P} \big[\sup_{t\in[0,T]}|\eta_t|^p\big] $ $<\infty$, and we denote its subset with continuous paths by $\S^p_c(\mathfrak{P})$;
    
    \item Let $\H^p(\mathfrak{P})$ denote all the progressively measurable processes $\eta: [0,T]\times\Omega\longrightarrow \R^{l\times k}$, s.t. $\sup_{\P\in\mathfrak{P}}\E^{\P} \mbrac{\big(\int_0^T|\eta_t|^2dt\big)^{\frac{p}{2}}} <\infty$;
    \item Let $\H^p_{\loc}(\mathfrak{P})$ denote all the progressively measurable processes $\eta: [0,T]\times\Omega\longrightarrow \R^{l\times k}$, s.t. there exists a sequence of stopping times $(\tau_n)_{n\in\N}$ with $ \tau_n\uparrow \infty$, and for each $n\in\N$, $\eta_{\cdot\wedge\tau_n}\in \H^p(\mathfrak{P})$;
    
    \item Let $\L^p\brac{\mathfrak{P}}$ denote all the $\F_T$-measurable random variables $\xi$, s.t. $\sup_{\P\in\mathfrak{P}}\E^{\P} \big[|\xi|^p\big]<\infty$;

    \item Let $\I^p(\mathfrak{P})$ denote all the càdlàg increasing processes $\eta: [0,T]\times\Omega\longrightarrow \R$, s.t. $\eta_0 = 0$ and $\sup_{\P\in\mathfrak{P}}\E^{\P}\big[ |\eta_T|^p\big] <\infty$, and we denote its subset with continuous paths by $\I^p_c(\mathfrak{P})$.
  
\end{itemize}
We also define $\mathbb{J}(\mathfrak{P})  :=\cup_{p>1}\mathbb{J}^p(\mathfrak{P}) $ for every $\mathfrak{P}$, and $\mathbb{J}  :=\cup_{p>1}\mathbb{J}^p(\P_0) $,  for $\mathbb{J}\in\set{\S,\S_c,\H,\H_{\loc},\L,\I,\I_c }$. \\

Throughout this paper we use $\x$ to denote an element in $\Omega$, while we use $x$ to denote an element in $\R^d$, and we use $x_{\cdot}$ to denote a constant path with value $x\in\R^d$. 

\section{Overview of Principal-Agent Problem}\label{sec:PA_problem}
\subsection{The Agent's problem}
Let $A$ and $B$ be two compact Polish spaces. Denote by $\mathbb{A}$ (resp.~$\mathbb{B}$) the collection of all $\mathbb{F}$-optional processes with values in $A$ (resp.~$B$). For $k\in\N$, we denote by $\Dc^k$ the collection of all functions
$$
\varphi:[0,T]\times\Omega\times A\times B\longrightarrow \R^k,
$$
such that, for every $(t,\x,a,b)\in [0,T]\times\Omega\times A\times B$, the process $\varphi(\cdot,a,b)$ is $\Fb$-optional and $\varphi_t(\x,\cdot)$ is continuous. In particular,
$
\varphi_t(\x,a,b)=\varphi_t(\x_{\cdot\wedge t},a,b)
$. Let $\mu\in\Dc^d$ be bounded. For every $(\alpha,\beta)\in\A\times\B$ , by Girsanov's theorem, there exists a unique probability measure $\P^{\alpha,\beta}$, s.t.\begin{equation*}
    dX_t = \mu_t(X,\alpha_t,\beta_t)dt+ dW^{\alpha,\beta}_t,
\end{equation*}
for some $\P^{\alpha,\beta}$-Brownian motion $W^{\alpha,\beta}$. We denote $\Pc:=\big\{\P^{\alpha,\beta}:\alpha\in\A,\beta\in \B\big\}$. By Novikov’s criterion, we know $\P^{\alpha,\beta}$ is equivalent to $ \P_0$ with Radon-Nikodym density
$$
\frac{\rd\P^{\alpha,\beta}}{\rd\P_0}\bigg|_{\F_t}
=\Ec_t \bigg(\int_{0}^{\cdot} \mu_s(X,\alpha_s,\beta_s)\cdot dX_s \bigg),
$$
where $
\Ec_t(D):=e^{D_t - \frac{1}{2}\langle D,D \rangle_t}$, and we denote by $\E^{\alpha,\beta}:=\E^{\P^{\alpha,\beta}}$, the expectation under  $\P^{\alpha,\beta}$. 

\begin{remark}\label{rem:Lq_equal_Lp}
    Let $\eta\in\F_T$ be a random variable. Then the following two assertions are equivalent: \begin{enumerate}
        \item 
    There exists $p>1$, s.t.  $\E^{\P_0}\big[|\eta|^p\big]<\infty$;
    \item There exists $q>1$, s.t. $\sup_{\P\in\Pc}\E^{\P}\big[|\eta|^q\big]<\infty.$ 
    \end{enumerate}
    Moreover, $\mathbb{J}(\Pc)  =\mathbb{J}(\P_0) $,  for $\mathbb{J}\in\set{\S,\S_c,\H,\H_{\loc},\L,\I,\I_c }$.\\
    
We only justify (1) implies (2), and we observe that the converse follows by applying the same argument to
$\frac{d\P_0}{d\P}\big|_{\F_t}$. For every $\P\in\Pc$, let $\Lambda^{\P}_t = \frac{\rd \P}{\rd\P_0}\big|_{\F_t} = \Ec_t\big(\int_0^{\cdot}\mu^{\P}_s\cdot dX_s\big)$, and we derive,  \begin{align*}
        |\Lambda^{\P}_t|^{p'}& = \Ec_t\bigg(\int_0^{\cdot}p'\mu^{\P}_s\!\cdot\! dX_s\bigg)e^{\frac{p'(p'-1)}{2}\!\int_0^t|\mu^{\P}_s|^2ds}\leq\Ec_t\bigg(\int_0^{\cdot}p'\mu^{\P}_s\!\cdot \!dX_s\bigg)e^{\frac{p'(p'-1)}{2}\|\mu^\P\|^2_{\infty}T},    \end{align*}
for every $p'>1$. Since $\mu$ is bounded, $\Ec\big(\int_0^{\cdot}p'\mu^{\P}_s\cdot dX_s\big)$ is a $\P_0$-uniformly integrable martingale. Therefore, let $p':=\frac{p}{p-q}$, by Doob's inequality, $M:=\sup_{\P\in\Pc}\E^{\P_0}\Big[\sup_{t\in[0,T]}|\Lambda^{\P}_t|^{p'}\Big]^{1/p'}<\infty$. Assume there exists $p>1$ s.t. $\E^{\P_0}\big[|\eta|^{p}\big]<\infty$. Pick $q\in(1,p)$, by Hölder's inequality, \begin{align*}
        \E^{\P}\big[|\eta|^{q} \big]=  \E^{\P_0}\big[|\eta|^{q}\Lambda^{\P}_T\big]\leq \E^{\P_0}\big[\big(\Lambda^{\P}_T\big)^{p'}\big]^{\frac{1}{p'}}
\E^{\P_0}\big[|\eta|^{p}\big]^{\frac{q}{p}}
        \le 
        M
        \E^{\P_0}\big[|\eta|^{p}\big]^{\frac{q}{p}}<\infty.
    \end{align*}
\end{remark}
 Let $\C_0\!:=\! \L\!\times\! \B$ denote the set of contracts. By Remark~\ref{rem:Lq_equal_Lp}, for every $\xi \in \L$, there exists $p>1$, s.t. 
\begin{align*}
   \sup_{(\alpha,\beta)\in \A\times \B} \E^{\alpha,\beta}\big[|\xi|^p\big]<\infty. 
\end{align*}

 We now introduce the cost function $f\in \Dc^1$, which satisfies the integrability assumption: \begin{align}
\label{eq:f_integ_assump}
\sup_{\alpha, \beta} \E^{\alpha,\beta}\bigg[\int_0^T|f_t(\alpha_t, \beta_t)|^p \rd t\bigg]<\infty, \text{ for some } p\in(1,\infty), 
\end{align}
and the discount factor is defined by:
$$
\Rc_t := e^{-\int_0^tr(s) ds},
~~\text{for some bounded measurable map}~~
r: [0,T]\longrightarrow \R.
$$
The agent's value function is:
 \begin{align}
 V^A(\xi,\beta) := \sup_{\alpha\in\A}J^A(\xi,\beta;\alpha),
 ~~\text{with}~~
    J^A(\xi,\beta;\alpha) :=\E^{\alpha,\beta}\bigg[\Rc_T\xi - \int_0^T\Rc_tf_t(\alpha_t,\beta_t)dt \bigg]. 
\end{align}
We denote by $\hat{\Ac} (\xi,\beta)$ the set of all the Agent's optimal responses under contract $(\xi,\beta)$, which will be shown to be non-empty under our assumptions. 
\subsection{The Principal's problem}
For a fixed positive constant $R$, we define the set of all admissible contracts by: \begin{equation}\label{eq:admissable_contract}
\C:=\big\{ (\xi,\beta)\in \C_0: V^A(\xi,\beta)\geq R\big\}.
\end{equation}
The restriction $V^A(\xi,\beta)\geq R$, known as the 'participation constraint', means the Agent will cooperate only if his expected payoff is at least his reservation utility $R$.\\

Let $g:\R\to\R$ be a non-decreasing function, interpreted as the Principal’s utility. Throughout this paper, we will not require the standard concavity condition on $g$. In addition to the critical non-decrease of $g$, we shall only need that its concave envelope $g^{\conc}$ be finite on $\R$. 

Let $\ell:\Omega\to\R$ be a measurable liquidation function, which represents the Principal’s cash inflow. We assume $\ell$ has polynomial growth, i.e. there exist $C,\kappa>0$, s.t. $$
|\ell(\x)|\leq C\big(1+\|\x\|^\kappa_{\infty}\big).
$$  
Under the standard favorable tie-breaking convention which states that, in case of multiple optimal responses, the Agent implements the most favorable one to the Principal, we define the Principal’s optimal contracting problem as:
\begin{align}\label{eq:principal_value}
V^P := \sup_{ (\xi,\beta)\in\C} \sup_{\alpha^{\star}\in \hat{\Ac} (\xi,\beta)}\E^{\alpha^{\star},\beta}\mbrac{g\brac{\ell-\xi}}.
\end{align}
Since $g^{\conc}<\infty$, there exists a constant $C>0$, s.t. $g^{\conc}(x)\leq C(1+\abs{x})$ for every $x\in\R$, so the expectation above is well-defined. By an immediate extension of \citet*[Theorem 4.2]{cvitanic2017dynamicprogrammingapproachprincipalagent}, we have
$$
V^P
= 
\sup_{y\geq R}
V^P(y),
~~\text{where}~~ V^P(y) := \sup_{\substack{Z\in \H\\
\beta \in \B}}
\;\sup_{\alpha^{\star}\in \hat{\Ac}\big(Y_T^{y,Z,\beta},\beta\big)}\E^{\alpha^{\star},\beta}\bigg[g\Big(\ell - Y_T^{y,Z,\beta}\Big)\bigg],
$$
and $Y=Y^{y,Z,\beta}$ is defined by the forward SDE:
\begin{align}\label{eq:SDE_of_Y}
    dY_t = Z_t\cdot dX_t+r(t)Y_tdt - H_t(Z_t, \beta_t)dt,\qquad Y_0 = y.
\end{align}
The Hamiltonian $H:[0,T]\times \Omega\times \R^d\times B\longrightarrow\R $ is defined as:\begin{align}\label{eq:hamiltonian}
    H_t(\x,z,b) := \sup_{a\in A}h_t(\x,z,a,b),\qquad h_t(\x,z,a,b):=z\cdot\mu_t(\x,a,b) - f_t(\x,a,b).
\end{align}
Since $A$ is compact and $a\mapsto h_t(\x,z,a,b)$ is continuous by assumption, by a standard measurable selection argument, we know the set $\hat{\Ac}\big(Y_T^{y,Z,\beta},\beta\big)=\big\{\alpha\in\A:\alpha_t\in \argmax_{a\in A} h_t(Z_t,a,\beta_t),~ dt\otimes d\P_0\text{-a.e.}\big\}$ is non-empty. Notice that $\hat{\Ac}\big(Y_T^{y,Z,\beta},\beta\big)$ depends on the contract only through $(Z,\beta)$, so we will simply denote it by $\hat{\Ac}(Z,\beta)$.  Moreover, as $g$ is non-decreasing, notice that the supremum over $y\ge R$ is achieved at $R$, so $$V^P = V^P(R),
~~\text{where}~~ V^P(y) := \sup_{\substack{Z\in \H\\
\beta \in \B}}
\;\sup_{\alpha^{\star}\in \hat{\Ac}(Z,\beta)}\E^{\alpha^{\star},\beta}\bigg[g\Big(\ell - Y_T^{y,Z,\beta}\Big)\bigg].$$

In the Markovian case where $\ell(\x)=\ell(\x_T)$ and $\varphi_t(\x,\cdot)=\varphi_t(\x_t,\cdot)$ for each $\varphi\in\{\mu,f\}$. Here we abuse the notation by writing $\ell$ and $\varphi$ for the induced scalar valued functions on $\R^d$ and $[0,T]\times\R^{d+1}$, respectively. In this case, under suitable assumptions (see \cite[Theorem 3.9]{cvitanic2017dynamicprogrammingapproachprincipalagent}), $V^P(y)=v(0,X_0,y)$, where the dynamic value function $v$ defined on $(t,x,y)\in[0,T]\times\R^d\times\R$ is characterized by the corresponding fully non-linear HJB equation:
\begin{align}
\left\{\begin{aligned}
&-\partial_t v
-\tfrac12 \Delta_{xx}v-r\,y\,\partial_yv
-F(\cdot,Dv,D^2v)
=0,
\qquad \text{on}~~[0,T)\times\R^d\times\R,\\[2pt]
&v(T,x,y)=g\big(\ell(x)-y\big),
\qquad (x,y)\in\R^d\times\R.
\end{aligned}\label{eq:HJB_fully_non_linear}
\right.
\end{align}
where for $p = (p_x,p_y)^{\top}$ and $\Theta=\left(\begin{array}{ll}
\Theta_{xx} & \Theta_{xy} \\
\Theta_{xy}^{\top} & \Theta_{yy}
\end{array}\right)$, $F$ is defined by:
\begin{align*}
F(t,x,p,\Theta)&:=\sup_{\substack{z\in\R^d\\ b\in B}}\ \sup_{q\in \partial_z H_t(x,z,b)}\Big[q\cdot\big(p_x+p_yz\big)
- H(t,x,z,b)p_y
+\tfrac12 |z|^2\Theta_{yy}
+ z\cdot\Theta_{xy}\Big].
\end{align*}
This standard method provides a solution to the Principal-Agent problem. However, there are still several pieces missing. First, the uniqueness for the fully nonlinear PDE \eqref{eq:HJB_fully_non_linear} in the viscosity sense is not obvious. Second, identifying the optimal contract requires further regularity of the solution to the PDE~\eqref{eq:HJB_fully_non_linear}.\\

In this paper, we provide a closed-form solution to this Principal-Agent problem, which connects to a backward SDE, and we also characterize the optimal contract. This is mainly due to the special structure of our setting, in which the discount rate $r$ is deterministic. In particular, we will show that, in this case, the fully non-linear PDE \eqref{eq:HJB_fully_non_linear} reduces to a semi-linear PDE.
\begin{remark}
    From Equation \eqref{eq:HJB_fully_non_linear}, we expect $v$  to be concave in $y$, otherwise $F$ will be infinite. Although this observation is at this stage only an intuition from the HJB equation in the Markovian setting, it will play a crucial role for the derivation of the main results of this paper.
\end{remark}

\section{BSDE Characterization of the Principal's Value Function}\label{sec:BSDE_chara}
\subsection{The main representation}
Consider a backward SDE:
\begin{align}\label{eq:BSDE_of_Gamma}
P_T=\ell
        ~~\text{and}~~
        d P_t =  Q_t\cdot dX_t  +r(t) P_tdt -\bar{H}_t( Q_t)dt, \quad \P_0\text{-a.s.} 
\end{align}
where $\bar{H}:[0,T]\times \Omega \times\R^d\longrightarrow\R$ is defined by:
$$
\bar{H}_t(\x,z):=\sup_{(a,b)\in A\times B}\big(z\cdot\mu_t(\x,a,b) - f_t(\x,a,b)\big) = \sup_{b\in B}H_t(\x,z,b) .
$$

Since $\mu$ is uniformly bounded, $\bar{H}$ is Lipschitz in $z$, uniformly in $(t,\x)$. Therefore, since $\ell\in\cap_{p>1}\L^p$, by classical $\mathbb{L}^p$ BSDE theory (see \citet*{BRIAND2003109}), Equation \eqref{eq:BSDE_of_Gamma} has a unique solution $( P,  Q)$ s.t. $ P\in\cap_{p\in(1,p_f)}\S_c^p$, $ Q\in\cap_{p\in(1,p_f)}\H^p$ for some $p_f>1$. Recall the notation $g^{\conc}$ for the concave envelope of $g$. Our first main result is the following explicit expression for the Principal's value function \eqref{eq:principal_value}. 

\begin{theorem}\label{thm:solve_PA}Assume $g$ is non-decreasing and $g^{\conc}<\infty$,  then 
    \begin{align*}
    V^P = g^{\conc}\big(( P_0 - R)e^{\int_0^Tr(s)ds}\big).
\end{align*}
Moreover, if $( P_0 - R)e^{\int_0^Tr(s)ds}\in \{g = g^{\conc}\}$, an optimal contract $(\xi^{\star},\beta^{\star})$ is given by: $$\xi^{\star} = Y_T^{R, Q,\beta^{\star}}=\ell+ (R -  P_0)e^{\int_0^Tr(s)ds},\qquad \beta^{\star}_t\in \arg\max_{b\in B}H_t( Q_t,b).$$
\end{theorem}

\begin{remark}
Although we obtain an elegant solution without imposing any concavity assumption on $g$, we are currently unable to identify the optimal contract in general. Example~\ref{eg:no_optimizer} below suggests that some form of concavity may be necessary for the existence of an optimal contract.
\end{remark}

\begin{example}\label{eg:pde_markovian}
Consider the Markovian case where $\ell(\x) = \ell(\x_T)$, and $\varphi_t(\x,\cdot) = \varphi(t,\x_t,\cdot)$ for each $\varphi\in\big\{\mu, f\big\}$.
For every $ (t,x)\in[0,T]\times \R^d$, let 
$$w(t,x):= \E^{\P_0}[ P_t|X_{ t} = x].
$$ 
By Theorem \ref{thm:solve_PA}, the fully non-linear PDE of $v$ is reduced to a semilinear PDE for $w$. By standard viscosity theory, $w$ is a viscosity solution to the following PDE: \begin{align}\label{eq:semi_linear_PDE_of_w}
-\partial_t w+rw-\tfrac{1}{2}\Delta w - \bar{H}(\cdot,\nabla w) = 0,
~\text{on}~[0,T)\times\R^d,
~~\text{and}~~
w\big|_{t=T} = \ell~\text{on}~\R^d.
\end{align}
We can easily verify that if $g^{\conc}\in C^{2}(\R)$, $w\in C^{1,2}([0,T]\times \R^d)$, then $v(t,x,y):= g^{\conc}\big((w(t,x) - y)e^{\int_t^Tr(s)ds}\big)$ solves the following PDE: 
\begin{align}
\left\{\begin{aligned}
&-\partial_t v(t,x,y)
-\tfrac12 \Delta_{xx}v(t,x,y)-r\,y\,\partial_yv(t,x,y) \\
&\hspace{1.9em}-
F(t,x,Dv(t,x,y),D^2v(t,x,y))
=0,
\qquad (t,x,y)\in[0,T)\times\R^d\times\R,\\[2pt]
&v(T-,x,y)=g^{\conc}\big(\ell(x)-y\big),
\qquad (x,y)\in\R^d\times\R,
\end{aligned}\label{eq:HJB_fully_non_linear_non_concav}
\right.
\end{align}

which is an extension of Equation \eqref{eq:HJB_fully_non_linear} when $g$ is not necessarily concave. The details of the computation are given in Appendix~\ref{app:compute}. 
\end{example}

\medskip
The following example illustrates a situation in which no optimal contract exists.

\begin{example}\label{eg:no_optimizer}
    Let $
    A = [-1,1]$, $B = \{b\}$, $T = 1$, $\mu_t(a) = a$, $ f_t(a) = \frac{1}{2}a^2$, $ \ell = 0$, and $ r = 0$. Then  
    $$
    a^{\star}(z)
    = -\ch{z <-1} +z \ch{|z| \leq1}+ \ch{z >1},
    ~\text{and}~
    d Y_t=Z_t d W^{\alpha^{\star}(Z)}_t+\frac{1}{2}|a^{\star}(Z_t)|^2 d t.
    $$
The Principal's value function reduces to $$
    V^P = \sup_{Z}\E^{\alpha^{\star}(Z)}\Big[g\Big(-R-\int_0^1Z_tdW_t^{\alpha^{\star}(Z)} - \frac{1}{2}\int_0^1|a^{\star} (Z_t)|^2dt\Big)\Big]. 
    $$
    Let $g(x) = x\mathds{1}_{[0,1)}(x)+\mathds{1}_{[1,+\infty)}(x)$, so $g^{\conc}\equiv 1$, and $\big\{g^{\conc} = g\big\}=[1,+\infty)$.  Since $P\equiv 0$, by Theorem \ref{thm:solve_PA}, $V^P$ has an optimizer if $R\leq -1$. In the current example, this is actually a necessary and sufficient condition. Indeed, when $R\leq -1$, $ Q = 0$ is the optimizer. On the other hand, when $R>-1$, we assume there exists an optimizer $Q$. Let
 $$
 J_t:=\int_0^t Q_sdW^{\alpha^{\star}(Q)}_s + \frac{1}{2}\int_0^t|\alpha^{\star} ( Q_s)|^2ds,
 $$  
   then we must have $g(-R - J_1) =1$ almost surely. This implies that $J_1\leq -R-1<0$ almost surely. However, by definition of $J$, $\E^{\alpha^{\star}(Q)}\big[J_1\big] = \frac{1}{2}\E^{\alpha^{\star}(Q)}\int_0^1|\alpha^{\star} ( Q_s)|^2ds\geq 0$, which is a contradiction. 

We can also see that, even if we enlarge the filtration by adding some randomness independent of $X$ (thus independent of $W$), the same proof in this example still works, i.e. there is no optimizer when $( P_0-R)e^{\int_0^Tr(s)ds}\notin \big\{g^{\conc} = g\big\}$.
\end{example}

\medskip
In order to prove Theorem~\ref{thm:solve_PA}, we first derive an a priori two-sided
estimate for the Principal's continuation value. We use $Y^{(t,y);Z,\beta}$ to denote the solution to Equation \eqref{eq:SDE_of_Y} on $[t,T]$, starting from $Y_t=y$ under the controls $(Z,\beta)$. Define $$
V_t(y):=\esssup_{\substack{Z\in \H \\
\beta \in \B}}\esssup_{\alpha^{\star}\in \hat{\Ac}(Z,\beta)}\E^{\alpha^{\star},\beta}\bigg[g\Big(\ell - Y_T^{(t,y);Z,\beta}\Big)\bigg|\F_t\bigg].
$$
\begin{lemma}\label{lem:bound_VP}For every $(t,y)\in[0,T]\times\R$,
    \begin{equation*}
g\brac[3]{( P_t - y)e^{\int_t^Tr(s)ds}}\leq V_t(y)\leq g^{\conc}\brac[3]{( P_t - y)e^{\int_t^Tr(s)ds}}.
    \end{equation*}
\end{lemma}

\begin{proof}
For simplicity, we present the proof for $t=0$. For an arbitrary $t\in[0,T]$, the same argument applies.    
\textit{Step 1.} For every $(Z,\beta)\in\H\times\B$, we define $$M^{y,Z,\beta}_t:=( P_t - Y^{y,Z,\beta}_t)e^{\int_t^Tr(s)ds},  ~~t\in[0,T].$$ We first prove, for every $\alpha^{\star}\in\hat{\Ac}\big(Z,\beta\big)$, $M^{y,Z,\beta}$ is a $\P^{\alpha^{\star},\beta}$-supermartingale. Let $\partial_z H(z,b)$ denote the subgradient of $H$ with respect to $z$, then $\mu_t(\alpha^{\star}_t(z,b),b)\in \partial_z H(z,b)$. Define
$$\mathscr{H}^{Z,\beta}_t:= H_t(Z_t,\beta_t) - \bar{H}_t( Q_t)+( Q_t-Z_t)\mu_t(\alpha^{\star}_t(Z_t,\beta_t),\beta_t).$$ By our assumption $A$ is compact and $(a,b)\mapsto h_t(\x,z,a,b)$ is continuous, by Berge's maximum theorem, we know $b\mapsto H_t(z,b)$ is continuous. Since $B$ is also compact, for every $z\in\R^d$, there exists  $\hat{b}\in\arg\max_{b\in B}H_t(z,b)$. By a standard measurable selection argument, there exists $\beta^\star\in\B$, s.t. $\beta^\star_t\in\arg\max_{b\in B}H_t( Q_t,b)$, $ dt\otimes d\P_0$-a.e. Since $H$ is convex in $z$, we have\begin{align*}
\mathscr{H}^{Z,\beta}_t\leq  H_t(Z_t,\beta_t) - H_t( Q_t,\beta_t)+( Q_t-Z_t)\cdot \mu_t(\alpha^{\star}_t(Z_t,\beta_t),\beta_t)\leq 0,
\end{align*}
and $\mathscr{H}^{ Q,\beta^\star}_t \!=\! 0$ for every $\beta^\star\!\in\!\B$ such that $\beta^{\star}_t\!\in\! \arg\max_{b\in B}H_t( Q_t,b)~ dt\!\otimes\!d\P_0$-a.e. By Itô's formula,\begin{align*}
        dM^{y,Z,\beta}_t = e^{\int_t^Tr(s)ds}\Big[( Q_t-Z_t)dW_t^{\alpha^{\star},\beta}+\mathscr{H}^{Z,\beta}_tdt\Big].
    \end{align*}
    Since $Z\in\H^p\big(\P^{\alpha^{\star},\beta}\big)$ for some $p>1$ and $ Q\in \cap_{q\in(1,p_f)}\H^q\big(\P^{\alpha^{\star},\beta}\big)$, $ Q - Z\in \H^{p\wedge \frac{p_f+1}{2}}\big(\P^{\alpha^{\star},\beta}\big)$. Notice that $r$ is bounded, we obtain $\big\{\int_0^te^{-\int_0^sr(s')ds'}( Q_s-Z_s)\cdot dW_s^{\alpha^{\star},\beta}\big\}_{t\in[0,T]}$ is a $\P^{\alpha^{\star},\beta}$-martingale, and therefore $M^{y,Z,\beta}$ is a $\P^{\alpha^{\star},\beta}$-supermartingale.

\medskip

\textit{Step 2.}  Since $g$ is non-decreasing, so is $g^{\conc}$. Moreover, $g\leq g^{\conc}$. For every $(Z,\beta)\in\H\times \B$, and $\alpha^{\star}\in\hat{\Ac}\big(Z,\beta\big)$, combine this with Jensen's inequality, we derive
$$
\E^{\alpha^{\star},\beta}\Big[g\big(M_T^{y,Z,\beta}\big)\Big]\leq \E^{\alpha^{\star},\beta}\Big[g^{\conc}\big(M_T^{y,Z,\beta}\big)\Big]\leq g^{\conc}\Big(\E^{\alpha^{\star},\beta}\big[M_T^{y,Z,\beta}\big]\Big) \leq g^{\conc}(M_0^{y,Z,\beta}).
$$
Since $M^{y,Z,\beta}_T = \ell - Y_T^{y,Z,\beta}$, we obtain
\begin{align*}
V_0(y)\leq g^{\conc}\Big(( P_0 - y)e^{\int_0^Tr(s)ds}\Big).
\end{align*}

\medskip

\textit{Step 3.} Notice that for every $\beta^{\star}\in \arg\max_{b\in B}H_t( Q_t,b)$, we have $dM_t^{y,  Q,\beta^{\star}} = 0$. Therefore
\begin{align*}
    V_0(y)\geq \E^{\alpha^{\star}( Q,\beta^{\star}),\beta^{\star}}\Big[g\big(M_T^{y,  Q,\beta^{\star}}\big)\Big] = g\Big(( P_0 - y)e^{\int_0^Tr(s)ds}\Big).
\end{align*}
\end{proof}

We next state the crucial concavity result of the Principal's continuation value function in $y$.

\begin{lemma}\label{lem:v_concave_in_y} Assume $g^{\conc}<\infty$, then the function $y\mapsto V_0(y)$ is concave.
\end{lemma}

\begin{proof}
    See Section \ref{sec:proof_of_concavity}.
\end{proof}

\begin{proof}[Proof of Theorem \ref{thm:solve_PA}] 
Let $x_R:=(P_0-R)e^{\int_0^T r(s)ds}$, by the left hand side inequality in Lemma~\ref{lem:bound_VP}, it follows from Lemma~\ref{lem:v_concave_in_y} that $V^P\geq g^{\conc}(x_R)$. Then equality follows from the right hand side inequality in Lemma~\ref{lem:bound_VP}. When $x_R\in\{g=g^{\conc}\}$, Step 3 of Lemma~\ref{lem:bound_VP} gives:
$
J^P(R,Q,\beta^\star)=g(x_R)=g^{\conc}(x_R)=V^P.
$
Therefore, the supremum is attained at $(Z,\beta)=(Q,\beta^\star)$, and the corresponding contract $(\xi^\star,\beta^\star)$ is optimal.
\end{proof}

\begin{remark}\label{rmk:relax_boundedness_mu}
The boundedness of $\mu$ is an important assumption which ensures the equivalence between the measures $\P^{\alpha,\beta}$, the Lipschitz continuity of $\bar{H}_t(x, z)$ in $z$, so that classical $L^p$ BSDE theory applies for $(P,Q)$, and to derive the following crucial uniform $L^p$ bound:
$$\sup_{n\in\N}\P^n(\Gamma) \leq \P_0(\Gamma)^{\frac{1}{p}}
\sup_{n\in\N}
\Big\|\frac{d\P^n}{d\P_0}\Big\|_{\L^q},
~~\text{for every}~\Gamma\in\F_T,
~~\frac{1}{p}+\frac{1}{q} = 1.
$$
We may formulate weaker assumptions on $\mu$ as long as the three last properties are preserved. We may assume for instance that
$ L_t^{\alpha,\beta}
        :=
        \mathcal E_t\brac[3]{
        \int_0^\cdot
        \mu_s(X,\alpha_s,\beta_s)\cdot dW_s}$
is a strictly positive uniformly integrable martingale, for all $(\alpha,\beta)\in \A\times \B$, and that $\sup_{(\alpha,\beta)\in\A\times\B}
        \E^{\P_0}\mbrac[3]{\brac[2]{L_T^{\alpha,\beta}}^q}<\infty$, for some $q>1$. As for the $\L^p-$backward SDE theory, one may replace it by another appropriate setting, see e.g. the quadratic setting in Section \ref{sec:cara} below.
\end{remark}

\subsection{Controlled volatility}
    Our approach can also be extended to the case where the volatility is controlled. We briefly indicate here why the same reduction mechanism still applies. We adopt the same framework as in Chiusolo and Hubert
\cite{chiusolo2025newapproachprincipalagentproblems} with unit discount factor. The output process satisfies 
\begin{equation*}
    dX_t = \sigma_t(X,\nu_t)\big(\mu_t(X,\nu_t)dt+ dW^{\nu}_t\big),\qquad \P^{\nu}\text{-a.s.}
\end{equation*}
for some probability measure $\P^{\nu}$ and some $\P^{\nu}$-Brownian motion $W^{\nu}$, and the control process $\nu\in\mathfrak{U}$, the set of $\mathbb{F}$-optional processes with values in a compact subset $U$ of a finite dimensional space. The Agent's and the Principal's value functions are given by
\begin{align*}
V^A(\xi):=\sup_{\nu\in \mathfrak{U}}\E^{\nu}\bigg[\xi - \int_0^Tf_t(X,\nu_t)dt \bigg],\qquad
V^P:=\sup_{\xi:V^A(\xi)\geq R}\E^{\nu^{\star}(\xi)}\Big[g(\ell - \xi)\Big],
\end{align*}
where $\nu^{\star}(\xi)$ is the Agent's optimal response, assumed to be unique for simplicity.
For every $y,Z,\Sigma$, we consider the forward SDE: $$
dY^{y,Z,\Sigma}_t = Z_t\cdot dX_t - H_t(Z_t,\Sigma_t)dt, \quad  \P^{\nu}\text{-a.s.}   \qquad Y^{y,Z,\Sigma}_0 = y, 
$$
for some $\nu\in U^\circ_t(\Sigma_t)$, where $$H_t(z,\Theta):=\sup_{ P\in U^\circ_t(\Theta)}\big[z\cdot\sigma_t( P)\mu_t( P) - f_t( P)\big],  \qquad U^\circ_t(\Theta):=\Big\{ P\in   U: \sigma_t( P)\sigma^{\top}_t( P) = \Theta\Big\}.$$
We have the convention that the supremum over an empty set is $-\infty$. Let $S_d$ denote the set of $d\times d$ positive definite matrices and $\mathcal{S}_d$ denote the set of $\mathbb{F}$-optional processes with values in $S_d$. Under the assumptions of
\citet*[Theorem~3.10]{chiusolo2025newapproachprincipalagentproblems},
the original problem is equivalent to the first-best reformulation $$
V^P(R)=\sup_{y\geq R}\;\sup_{\Sigma\in \mathcal{S}_d}\;\sup_{Z\in \H}\E^{\nu^{\star}(Z,\Sigma)}\Big[g\big(\ell - Y_T^{y,Z,\Sigma}\big)\Big].
$$
Let $\mathfrak{U}^{\circ}(\Sigma):= \big\{\nu\in  \mathfrak{U}:\nu_t\in U_t^{\circ}(\Sigma_t),~\P^{\nu}\text{-a.s.} \big\}$. Now for every $\Sigma\in\mathcal{S}_d$ such that $  \mathfrak{U}^{\circ}(\Sigma)\neq \emptyset$, consider the BSDE:
$$
d P^{\Sigma}_t =  Q^{\Sigma}_t\cdot dX_t - H_t\big( Q^{\Sigma}_t,\Sigma_t\big)dt,    \qquad  P^{\Sigma}_T = \ell,\qquad \P^{\nu}\text{-a.s.}
$$
for some $\nu\in U^\circ_t(\Sigma_t)$.
Then the same argument as in Theorem \ref{thm:solve_PA} yields: $$\sup_{Z\in\H}\E^{\nu^{\star}(Z,\Sigma)}\Big[g\big(\ell - Y_T^{y,Z,\Sigma}\big)\Big] = g^{\conc}( P^{\Sigma}_0 - y),
~~\text{and then}~~
V^P(R)=g^{\conc}\bigg(\sup_{\Sigma\in\mathcal{S}_d:  \mathfrak{U}^{\circ}(\Sigma)\neq\emptyset} P^{\Sigma}_0 - R\bigg),$$ 
by the non-decrease of $g$.
Therefore, even in the presence of volatility control, the Principal's value can still be characterized through the supremum of a family of BSDEs. In particular, if $  U = A\times B$, $\nu =(a,b)$, $\mu(\cdot,\nu) = \mu(\cdot,a)$ and $\sigma(\cdot,\nu) = \sigma(\cdot,b)$, this will coincide with the framework in \citet*{cvitanic2017dynamicprogrammingapproachprincipalagent}. In this case, consider the BSDE:
$$
d P^{\beta}_t =  Q^{\beta}_t\cdot dX_t - \tilde{H}_t\big( Q^{\beta}_t,\beta_t\big)dt,    \qquad  P^{\beta}_T = \ell,\qquad \P^{\nu}\text{-a.s.}
$$
where $\tilde{H}_t(z,b):= \sup_{a\in A}\big(z\cdot\sigma_t(b)\mu_t(a) - f_t(a,b)\big)$. Notice that $H_t(z,\Sigma_t) =\sup_{\beta\in  \mathfrak{U}^{\circ}(\Sigma)} \tilde{H}_t(z,\beta_t)$, we have $$
 P_0^{\Sigma}=\sup_{\beta\in  \mathfrak{U}^{\circ}(\Sigma)} P_0^{\beta},
~~\text{and therefore}~~
V^P(R)=g^{\conc}\Big(\sup_{\beta\in B} P^{\beta}_0 - R\Big).
$$

This is consistent with the second-order BSDE formulation in \cite{cvitanic2017dynamicprogrammingapproachprincipalagent}. In the Markovian case, similar to the previous example \ref{eg:pde_markovian}, if we let $u(t,x) = \E_{t,x}\big[\sup_{\beta} P^{\beta}_t\big|\F_t]$, under appropriate assumptions, we can obtain the corresponding HJB equation:
\begin{align}\label{eq:reduced_fully_non_linear_linear_PDE_of_w}
-\partial_t u- \mathcal{H}\big(\cdot,D u,D^2 u\big) = 0,
~~\text{on}~[0,T)\times\R^d,~~
\text{and}~u(T,\cdot) = \ell
~~\text{on}~\R^d,
\end{align}
where $\mathcal{H}\big(t,x,p,\Theta\big):=\sup_{a\in A,b\in B}\Big[\sigma(t,x,b)\mu(t,x,a)\cdot p^{\top} +\tfrac{1}{2}\tr{(\sigma\sigma^{\top})(t,x,b)\Theta}- f(t,x,a,b)\Big]$. This equation remains a fully non-linear PDE, but it is already much simpler than the original HJB equation provided in \cite[Theorem 3.9]{cvitanic2017dynamicprogrammingapproachprincipalagent}.

\subsection{Concavity of the Principal's value function}
\label{sec:proof_of_concavity}
This subsection is dedicated to the proof of Lemma \ref{lem:v_concave_in_y}. Our argument follows the same spirit as Method 2 in the proof of
Proposition~4.5 by~\citet*{i̇şeri2025setvaluedhamiltonjacobibellmanequations}. As $V_0$ is locally bounded by Lemma~\ref{lem:bound_VP}, it suffices to show that:
$$
V_0(\bar{y})\geq \tfrac{1}{2}\big(V_0(y_1)+V_0(y_2)\big),
~\text{for all fixed}~
y_1>y_2,~\bar{y}:=\frac{y_1+y_2}{2}.
$$
For each $n\in\N$, let $Z^n$ be the constant process defined by $Z_t^n:=(n,0,\cdots,0)$, and let $\beta^0$ be the constant process identically equal to a fixed $b_0\in B$. Take any fixed $\alpha^{n}\in \hat{\Ac}(Z^n,\beta^0)$, we denote $Y^n:= Y^{\bar{y},Z^n,\beta^0}$, $f^n_s:= f_s(\alpha^{n}_s,\beta^0_s)$, $\P^n:=\P^{\alpha^{n},\beta^0}$ and $W^n:=W^{\alpha^{n},\beta^0,1}$. We define the exit time: 
\begin{align}\label{eq:tau_1_tau_2}
\tau^n:= \tau^n_1\wedge \tau^n_2,
~~\text{where}~~
    \tau^n_1:=T\wedge\inf\{s\geq 0:Y^n_s\geq y_1\} 
    ~~\text{and}~~
     \tau^n_2:=T\wedge\inf\{s\geq 0:Y^n_s\leq y_2\}.
\end{align}
\begin{lemma}\label{lm:estimate_tau_n}
$\lim_{n\to\infty}\tau^n=0$, a.s. and $\lim_{n\to\infty}\P^n[ \tau^n =  \tau^n_1<\theta]=\lim_{n\rightarrow\infty} \P^n[\tau^n = \tau^n_2<\theta] =\frac{1}{2}$, for any stopping time $\theta>0$ a.s.
\end{lemma}
\begin{proof}
{\bf 1.} We start with the first claim. By assumption on $f$, we have $\sup_{n\in\N}\E^{\P^n}\Big[\int_0^T|f^n_t|^pdt\Big]<\infty$, for some $p>1$. Let $\delta = \tfrac{y_1-y_2}{2}$, and recall that $Y^n_s - \bar{y}=\int^s_0\big(f^n_{u} + r(u)Y^n_u\big)du + nW^n_s$, then we estimate for every $t\in (0,T)$,
    \begin{align*}
        \P^n[\tau^n > t] &= \P^n\bigg[\sup_{s\in[0,t]}|Y^n_s - \bar{y} |< \delta\bigg]\\
        &\leq \P^n\bigg[\Big\{ n\sup_{s\in[0,t]}\big|W^n_s \big|-\int^{t}_0|f^n_{u}|du - \|r\|_{\infty}\int^{t}_0|Y^n_u|du < \delta\Big\}
        \cap\Big\{\sup_{s\in[0,t]}|Y^n_s - \bar{y} |< \delta\Big\}\bigg] \\
        &\leq \P^n\bigg[ n\sup_{s\in[0,t]}\big|W^n_s \big|\leq a_n\bigg];
        ~~\text{where}~~a_n:= \int^{t}_0|f^n_{u}|du + t \|r\|_{\infty}(\delta+\abs{\bar{y}}) + \delta.
    \end{align*}
Notice that $C_A:=\sup_{n\in\N}\E^{\P^n}[a_n^p]<\infty$, then we continue the last inequality for every $\gamma>0$:
\begin{align*}
\P^n[\tau^n \!>\! t]
\le
    \P^n\bigg[ \sup_{s\in[0,t]}\big|W^n_s \big|\leq\frac{a_n}{n} \bigg]&\leq  \P^n\bigg[ \sup_{s\in[0,t]}\big|W^n_s \big|\leq \gamma\bigg] +\P^n\bigg[ \gamma\leq\frac{a_n}{n}\bigg]
    \leq \frac{4}{\pi}e^{-\frac{C}{\gamma^2}}
    +\frac{C_A}{(n\gamma)^p},
\end{align*}
where $C = \frac{\pi^2t}{8}$. Now let $\gamma = n^{-\epsilon}$, where $\epsilon>0$ is small enough such that $(1-\epsilon)p>1$. Then:  
\begin{align*}
      \P^n[\tau^n > t]
      \leq \frac{4}{\pi}e^{-Cn^{2\epsilon}} + \frac{C_A}{n^{p(1-\epsilon)}}
      \le
      Cn^{-p(1-\epsilon)},
\end{align*}
for some constant $C$ independent of $n$, which implies by Remark \ref{rem:Lq_equal_Lp} that $\P_0[\tau^n > t]\le C'n^{-p'}$ for some constant $C'$ independent of $n$ and $p'>1$. Then, $\sum_{n=1}^{\infty}\P_0[\tau^n> t]<\infty$, and we conclude by the Borel-Cantelli lemma that $\tau^n\longrightarrow0$ a.s.

\medskip
\noindent {\bf 2.} For every $\eta\in(0,\frac{y_1-y_2}{2})$, let $G^n_{\eta}:=\big\{\omega\in\Omega:\,\sup_{t\in[0,\tau^n]}\big|\int_0^{t}\big(f^n_s+r(s)Y^n_s\big)ds\big|> \eta\big\}$. By the first step of the current proof, $\lim_{n\to\infty}\P^n[G_\eta^n]=0$ for all $\eta>0$. Denote $
    \sigma^n_a:=\inf\big\{s\geq 0: nW^n_s+\bar{y}=a\big\}$, for all $n\in\N$, $a\in\R$, and let us show that 
    \begin{equation}
      \label{eq:set_subset_G}
    \big\{\sigma^n_{y_1-\eta}>\sigma^n_{y_2-\eta}\big\}\cap \big\{\tau^n = \tau^n_1,\tau^n<T\big\}\subset G^n_\eta.
      \end{equation}
    Take $\omega\in\big\{\sigma^n_{y_1-\eta}>\sigma^n_{y_2-\eta}\big\}\cap \big\{\tau^n = \tau^n_1,\tau^n<T\big\}$, assume $\omega\notin G^n_\eta$, then $\sup_{t\in[0,\tau^n]}\Big|\int_0^{t}\big(f^n_s(\omega)+r(s)Y^n_s(\omega)\big)ds\Big|\leq \eta$. As $Y_{\tau^n}^n(\omega)=y_1$, we have $$\bar{y}+ nW^n_{\tau^n} (\omega)=Y_{\tau^n}^n(\omega) -\int_0^{\tau^n}\big(f^n_s(\omega)+r(s)Y^n_s(\omega)\big)ds\geq y_1-\eta.$$
    Since $\bar{y}<y_1-\eta $, continuity of $W^n$ gives 
$\tau^n(\omega)\geq \sigma_{y_1-\eta}^n(\omega) >\sigma^n_{y_2-\eta}(\omega)  $. On the other hand, since $\omega\notin G_\eta^n$, we obtain $$
Y_{\sigma^n_{y_2-\eta}}^n(\omega)=\bar{y}+n W_{\sigma^n_{y_2-\eta}}^n(\omega)+\int_0^{\sigma^n_{y_2-\eta}}\big(f^n_s(\omega)+r(s)Y^n_s(\omega)\big)ds \leq y_2-\eta+\eta=y_2.
$$
This implies that $\sigma^n_{y_2-\eta}(\omega)\geq \tau^n_2(\omega)$, and therefore $\tau^n(\omega)> \tau^n_2(\omega)$, which is a contradiction.\\
    
We deduce from~\eqref{eq:set_subset_G} that $
    \big\{\tau^n = \tau^n_1,\tau^n<T\big\}\subset\big\{\sigma^n_{y_1-\eta}\leq\sigma^n_{y_2-\eta}\big\}\cup  G^n_\eta
    $, and then
    \begin{align*}
        \P^n[\tau^n = \tau^n_1,\tau^n<T]&\leq \P^n\big[\sigma^n_{y_1-\eta}\leq\sigma^n_{y_2-\eta}\big]+\P^n[  G^n_\eta]
        = \tfrac{1}{2}+\tfrac{\eta}{y_1-y_2}+\P^n[  G^n_\eta].
    \end{align*}
Thus, letting $n\rightarrow\infty$ and then $\eta\rightarrow0^+$, we obtain $\limsup_{n\rightarrow\infty} \P^n[\tau^n = \tau^n_1,\tau^n<T] \leq \tfrac{1}{2}$ and by symmetry, $\limsup_{n\rightarrow\infty} \P^n[\tau^n = \tau^n_2,\tau^n<T] \leq \tfrac{1}{2}$.  As $\lim_{n\rightarrow\infty} \P^n[\tau^n = T] = 0$
by the first step of this proof, we conclude that
$\lim_{n\rightarrow\infty} \P^n[\tau^n = \tau^n_1] = \lim_{n\rightarrow\infty} \P^n[\tau^n = \tau^n_2] = \tfrac{1}{2}.$

\medskip
\noindent {\bf 3.} For any stopping time $\theta>0$, since $\tau^n\to0$ a.s. we have $\ch{\theta\leq\tau^n}\longrightarrow0
$ a.s. Hence $\lim_{n\rightarrow \infty}\P_0(\theta\leq\tau^n) = 0$. By Lemma~\ref{lm:sup_P0_equal_sup_Pn}, we obtain $\lim_{n\rightarrow \infty}\P^n(\theta\leq\tau^n) = 0$. Therefore, for $i\in\set{1,2}$
\begin{align*}
0
&\leq \mathbb P^n(\tau^n=\tau_i^n)
-\mathbb P^n(\tau^n=\tau_i^n<\theta)= \mathbb P^n(\tau^n=\tau_i^n\geq \theta)
\leq \mathbb P^n(\theta\leq\tau^n)
\xrightarrow{n\rightarrow\infty}0.
\end{align*}
So we conclude that $\lim_{n\rightarrow\infty} \P^n[\tau^n = \tau^n_1,\tau^n<\theta] =\lim_{n\rightarrow\infty} \P^n[\tau^n = \tau^n_2,\tau^n<\theta] = \tfrac{1}{2}$.
\end{proof}

\begin{lemma}\label{lm:sup_P0_equal_sup_Pn}
    Let $\set{A^i}_{i\in\N}\subset\mathcal{F}_T$. Then $\lim_{i\rightarrow\infty}\P_0[A^i] = 0$ if and only if $\lim_{i\rightarrow\infty}\displaystyle\sup_{n\in\N}\P^n[A^i]=0$. 
\end{lemma}
\begin{proof}
Notice that $\P^{n}[A^i] = \E^{\P_0}\big[\mathds{1}_{A^i}\frac{\rd \P^n}{\rd \P_0}\big]\leq \P_0[A^i]^{\frac{1}{2}}\E^{\P_0}\big[\big(\frac{\rd \P^n}{\rd \P_0}\big)^2\big]^{\frac{1}{2}}$. This proves one direction as $\E^{\P_0}\Big[\big(\frac{\rd \P^n}{\rd \P_0}\big)^2\Big]$ is uniformly bounded by the boundedness of $\mu$. The other direction is obvious. 
\end{proof}
\begin{lemma}\label{lm:weak_right_continuous}
Let $\set{\theta_n}_{n\in\N}$ be a sequence of stopping times such that $\lim_{n\to\infty}\theta_n = 0$, a.s. then
$$ 
\liminf_{n\to\infty}V_{\theta_n}(y)\ge V_0(y+)
\qquad\text{a.s.}
$$
\end{lemma}

\begin{proof}
Fix $y\in\R$. For each $m\in\N$, choose
$(\hat Z^m,\hat\beta^m)\in\H\times\B$ and
$\hat\alpha^m\in
\hat{\Ac}(\hat Z^m,\hat\beta^m)$ such that $\E^{\mathbb Q^m}\big[g(\ell-\hat Y_T^m)\big]
\geq V_0\big(y+\tfrac{1}{m}\big)-\tfrac{1}{m}
$, where $
\mathbb Q^m:=\P^{\hat\alpha^m,\hat\beta^m}$, $\hat Y^m:=
Y^{(0,y+\frac{1}{m});\hat Z^m,\hat\beta^m}$. Since $g(x)\leq g^{\conc}(x)\leq C(1+|x|)$, we have
\begin{align*}
\E^{\mathbb Q^m}
\Big[\big(g(\ell-\hat Y_T^m)\big)^+\Big]
&\leq
C\Big(1+\E^{\mathbb Q^m}\Big[|\ell|+\big|\hat Y_T^m\big|\Big]\Big)
<+\infty.\\
\E^{\mathbb Q^m}
\big[\big(g(\ell-\hat Y_T^m)\big)^-\big]
&=
\E^{\mathbb Q^m}
\Big[\big(g(\ell-\hat Y_T^m)\big)^+\Big]
-\E^{\mathbb Q^m}\Big[g\big(\ell-\hat Y_T^m\big)\Big]\\
&\leq
\E^{\mathbb Q^m}
\Big[\big(g\big(\ell-\hat Y_T^m\big)\big)^+\Big]
-V_0\big(y+\tfrac{1}{m}\big)+\tfrac{1}{m}
<+\infty.
\end{align*}
Therefore $g(\ell-\hat Y_T^m)\in \L^1$. Let
$
M_t^m:=
\E^{\mathbb Q^m}\big[g(\ell-\hat Y_T^m)\mid\F_t\big]
$, $ t\in[0,T],
$
so $M^m$ is a uniformly integrable $\mathbb Q^m$-martingale.
In particular, we take its càdlàg version. On $\{\hat Y_{\theta_n}^m\geq y\}$, we derive
$$
 V_{\theta_n}(y)\geq \E^{\mathbb Q^m}
\Big[g\big(\ell - Y_T^{(\theta_n,y);\hat Z^m,\hat\beta^m}\big)\big|\F_{\theta_n}\Big]
= \E^{\mathbb Q^m}
\Big[g\big(\ell - \hat Y_T^m-
(y-\hat Y_{\theta_n}^m)
e^{\int_{\theta_n}^T r(s)ds}\big)\big|\F_{\theta_n}\Big]\geq M_{\theta_n}^m.
$$
Since $\hat Y_{0}^m=y+\tfrac{1}{m}>y$, $\theta_n\rightarrow 0$ a.s. implies $\ch{\hat Y_{\theta_n}^m\ge y}\rightarrow 1$ and $\ch{\hat Y_{\theta_n}^m< y}\rightarrow 0$, a.s. Notice that Lemma~\ref{lem:bound_VP} implies $V_{t}(y)$ is locally bounded, therefore
\begin{align*}
\liminf_{n\rightarrow\infty}V_{\theta_n}(y) 
&\geq \liminf_{n\rightarrow\infty}M^m_{\theta_n}\ch{\hat Y_{\theta_n}^m\ge y}  +\liminf_{n\rightarrow\infty}V_{\theta_n}(y)\ch{\hat Y_{\theta_n}^m< y}=M_0^m\geq V_0\big(y+\tfrac{1}{m}\big)-\tfrac{1}{m},\qquad a.s.
\end{align*}
Notice that $V_0$ is non-increasing, so $V_0(y+)$ exists. We pass $m$ to infinity and complete the proof.
\end{proof}

\begin{proof}[Proof of Lemma \ref{lem:v_concave_in_y}]
   Fix $K>0$, let $\theta_K:=\inf\set{t\geq 0: \abs{ P_t -  P_0} = K}\wedge T$. By the dynamic programming principle, we know \begin{align*}
        V_0(\bar{y})&\geq \E^{\P^n}\Big[V_{\tau^n \wedge \theta_K}\big(Y_{\tau^n\wedge \theta_K}^{n}\big)\Big]\\
        & = \E^{\P^n}\Big[V_{\tau^n_1}(y_1)\ch{\tau^n = \tau^n_1<\theta_K}\Big]
+\E^{\P^n}\Big[V_{\tau^n_2}\big(y_2\big)\ch{\tau^n = \tau^n_2<\theta_K}\Big]
        +\E^{\P^n}\Big[V_{\theta_K}(Y_{ \theta_K}^{n})\ch{\theta_K\leq \tau^n }\Big].
    \end{align*}
    Since $g^{\conc}<\infty$, there exists $C>0$, s.t. $g(x)\leq g^{\conc}(x) \leq C(1+|x|)$. Let $M_K:=\Big|g\brac[2]{ -|P_0-K - y_1|e^{T \|r\|_{\infty}}}\Big|$, by Lemma~\ref{lem:bound_VP}, we know, on $\set{\theta_K\leq \tau^n}$, for every $t\in[0,\theta_K]$,\begin{align}\label{eq:def_Mk}
        V_{t}(Y_{ t}^{n})\geq g\brac[3]{( P_{t} - Y_{ t}^{n})e^{\int_t^T{r(s)}ds}}&\geq g\brac[3]{( P_0-K - y_1)e^{\int_t^T{r(s)}ds}}\geq -M_K.
    \end{align}
     Notice that $\set{\theta_K\leq \tau^n}\subset\set{\theta_K\leq \varepsilon}\cup \set{\varepsilon\leq \tau^n}$ for every $\varepsilon>0$, therefore, 
     $$
     \E^{\P^n}\Big[V_{\theta_K}(Y_{ \theta_K}^{n})\ch{\theta_K\leq \tau^n }\Big]
     \geq 
     -M_K\brac[4]{\sup_{n\in\N}\P^n\mbrac{\theta_K\leq \varepsilon}+\P^n\mbrac{\varepsilon\leq \tau^n}}
     \xrightarrow[n\to\infty]{}
     -M_K\sup_{n\in\N}\P^n\mbrac{\theta_K\leq \varepsilon},
     $$
     by Lemma \ref{lm:estimate_tau_n}. Since $ P$ has continuous paths, we have $\lim_{\varepsilon\rightarrow 0^+}\P_0[\theta_K\leq\varepsilon] = 0$, and it follows from lemma~\ref{lm:sup_P0_equal_sup_Pn} that $\lim_{\varepsilon\rightarrow 0^+}\sup_{n\in\N}\P^n\mbrac{\theta_K\leq\varepsilon} = 0$, and we derive from the last inequality that $\liminf_{n\rightarrow\infty}\E^{\P^n}\Big[V_{\theta_K}(Y_{ \theta_K}^{n})\ch{\theta_K\leq \tau^n }\Big] \geq 0$. Notice that $x\geq a-(a-x)^+$, we derive:
\begin{align*}
&\mathbb E^{\P^n}\left[
V_{\tau^n}(y_1)
\ch{\tau^n=\tau_1^n<\theta_K}
\right]\\
&\quad\ge
V_0(y_1+)\,\P^n(\tau^n=\tau_1^n<\theta_K)
-\mathbb E^{\P^n}\left[
\bigl(V_0(y_1+)-V_{\tau^n}(y_1)\bigr)^+
\ch{\tau^n=\tau_1^n<\theta_K}
\right]\\
&\quad=
V_0(y_1+)\,\P^n(\tau^n=\tau_1^n<\theta_K)
-\mathbb E^{\P_0}\!\left[
\tfrac{d\P^n}{d\P_0}
\bigl(V_0(y_1+)-V_{\tau^n}(y_1)\bigr)^+
\ch{\tau^n=\tau_1^n<\theta_K}
\right]\\
&\quad\ge
V_0(y_1+)\,\P^n(\tau^n=\tau_1^n<\theta_K)
\\
&\qquad-\bigg(\sup_n \E^{\P_0}\Big[\big(\tfrac{d\P^n}{d\P_0}\big)^2\Big]^{\frac12}\bigg)
\mathbb E^{\P_0}\!\Big[
\bigl((V_0(y_1+)-V_{\tau^n}(y_1))^+\bigr)^2
\ch{\tau^n=\tau_1^n<\theta_K}
\Big]
^{\frac12}.
\end{align*}

We apply Hölder's inequality to obtain the last inequality. By Lemma~\ref{lm:weak_right_continuous} we know $(V_0(y_1+)-V_{\tau^n}(y_1))^+$ almost surely converges to $0$. Notice that $(V_0(y_1+)-V_{\tau^n}(y_1))^+\leq V_0(y_1+)+M_K $ on $\{\tau^n=\tau_1^n<\theta_K\}$, we apply dominated convergence theorem and Lemma~\ref{lm:estimate_tau_n} to derive
\begin{align*}
\liminf_{n\rightarrow\infty}\E^{\P^n}\Big[V_{\tau^n_1}(y_1)\ch{\tau^n = \tau^n_1<\theta_K}\Big]&\geq \frac{1}{2} V_0(y_1+). 
\end{align*}

By symmetry we obtain $\liminf_{n\rightarrow\infty} \E^{\P^n}\big[V_{\tau^n_2}(y_2)\mathds{1}_{\{\tau^n = \tau^n_2,\tau^n<\theta_K\}}\big] \geq \frac{V_0(y_2+)}{2}$, combine with previous estimation we derive $$
V_0(\bar{y})\geq \frac{V_0(y_1+)+V_0(y_2+)}{2}.
$$
Since $g$ is non-decreasing, $V_0$ is non-increasing, so for every $h>0$,
$$
V_0(\bar{y}+h)\geq \frac{1}{2}\big(V_0((y_1+h)+)+V_0((y_2+h)+)\big)\geq \frac{1}{2}\big(V_0(y_1+2h)+V_0(y_2+2h)\big).
$$
Let $h\rightarrow0$, we obtain $
V_0(\bar{y}+)\geq \frac{1}{2}\big(V_0(y_1+)+V_0(y_2+)\big)$. By Lemma~\ref{lem:bound_VP}, $y\mapsto V_0(y+)$ is locally bounded, so it is concave and continuous. Since $V_0$ is non-increasing, for every $h>0$, we have
$$
V_0(y+)\leq V_0(y)\leq V_0((y-h)+).
$$
Let $h\rightarrow0$, by continuity of $y\mapsto V_0(y+)$, we obtain $V_0(y+)= V_0(y)$, and therefore complete the proof. 
\end{proof}

\section{Principal--Agent problem under Exponential Utility}\label{sec:cara}
 It is well known in the economics literature that, when the Agent is risk-neutral, the second-best Principal Agent problem reduces to the first-best problem; see e.g. \citet*[Proposition 3.2]{nizar_random_horizon_principal_agent}. Our objective in this section is to show that the explicit solution in Theorem~\ref{thm:solve_PA} is not related to the Agent's risk-neutrality.  To illustrate this fact, we shall show that the same method also applies when both the Agent and the Principal have exponential utility functions, a situation where the second-best problem no longer coincides with the first-best problem.
 
\subsection{Problem Formulation}
    In this section we additionally assume $f$ and $\ell$ are bounded. We consider the case where the Agent and the Principal have exponential utility functions, i.e. $g_A(y) = -e^{-\eta_A y}$, and $g(y) = -e^{-\eta_P y}$, $\eta_A,\eta_P>0$. For simplicity, we omit the discount factor although, as in the previous section, all results extend to the case of deterministic discount rate. The Agent's value function is: 
    \begin{align}
    V^A_{\cara}(\xi,\beta) &:=\sup_{\alpha\in\A}\E^{\alpha,\beta}\bigg[g_A\bigg(\xi - \int_0^Tf_t(\alpha_t,\beta_t)dt\bigg) \bigg],
\end{align}
for all contracts $(\xi,\beta)\in \Xi_{\cara}\times\B$, where
$$
\Xi_{\cara} := \big\{\xi\in\F_T : \E^{\P_0}\big[e^{-\eta_A' \xi}+e^{\eta_P' \xi}\big]<\infty, \text{ for some }\eta_A'>\eta_A,\eta_P'>\eta_P\big\}.
$$
To incorporate the Agent's participation constraint, we introduce the set of admissible contracts:
$$
\C_{\cara}:= \big\{(\xi,\beta)\in\Xi_{\cara}\times\B:V^A_{\cara}(\xi,\beta)\geq g_A(R)\big\},
$$
Note that $\C_{\cara}$ is always nonempty. We use $\hat{\Ac}^{\cara}(\xi,\beta)$ to denote the set of Agent's optimal responses under contract $(\xi,\beta)$, which is non-empty. The Principal's value function is then given by:
\begin{equation}\label{eq:cara_principal_value}
V^P_{\cara} = \sup_{\substack{(\xi,\beta)\in\C_{\cara}}} \sup_{\alpha^{\star}\in\hat{\Ac}^{\exp}(\xi,\beta)}\E^{\alpha^{\star},\beta} \big[g(\ell - \xi)\big].
\end{equation}
Since $\ell$ and $f$ are bounded, by Hölder's inequality we know both the expectation in the value function of the Agent and the Principal are finite. Recall the Hamiltonian $H$ of \eqref{eq:hamiltonian}, and define for all $(y,Z)\in\R\times \H_{\loc}$, the forward dynamics:
\begin{align}\label{eq:SDE_of_Y_exponential}
    dY_t = Z_t\cdot dX_t+ \tfrac{\eta_A}{2}|Z_t|^2dt - H_t(Z_t, \beta_t)dt,\qquad Y_0 = y,
\end{align}
Similar to \citet*{Euch_market_regulation}, for every $\beta\in\B$, we define \begin{align*}
\Zc^{\beta}:&= \bigg\{Z\in\H_{\loc}:\E^{\P_0}\bigg[\sup_{t\in[0,T]}e^{-\eta_A' Y_t^{0,Z,\beta}}+\sup_{t\in[0,T]}e^{\eta_P' Y_t^{0,Z,\beta}}\bigg]<\infty, \text{ for some } \eta_A'>\eta_A,\eta_P'>\eta_P\bigg\}.
\end{align*}

\begin{remark}
    In \cite{Euch_market_regulation}, the authors study a Principal--Agent problem with exponential utilities and a jump-diffusion controlled state process. They consider the larger admissible set
    \begin{align*}
\bar{\Zc}^{\beta}:&= \bigg\{Z\in\H_{\loc}:\E^{\P_0}\bigg[\sup_{t\in[0,T]}e^{-\eta_A' Y_t^{0,Z,\beta}}+e^{\eta_P' Y_T^{0,Z,\beta}}\bigg]<\infty,\text{ for some } \eta_A'>\eta_A,\eta_P'>\eta_P\bigg\}.
\end{align*}
Our restriction to the class $\Zc$ is needed to justify the supermartingale argument in Theorem~\ref{prop:super_Mg_cara}. We will see in Remark \ref{rmk:Z_beta_equal_bar_Z_beta} that $\Zc$ and $\bar{\Zc}$ are actually identical under our setting.
\end{remark}

Use $\hat{\Ac}^{\exp}(Z,\beta)$ to denote $\hat{\Ac}^{\exp}\big(Y_T^{y,Z,\beta},\beta\big)$, similar to Equation \eqref{eq:principal_value}, it is standard to reduce the problem into a standard stochastic control problem:
\begin{proposition}\label{prop:reduce_to_standard_control_cara} In the above setting,
$\displaystyle
V^P_{\cara}= 
\sup_{\beta\in\B}\;
\sup_{Z\in \Zc^{\beta}
}\;
\sup_{\alpha^{\star}\in\hat{\Ac}^{\exp}(Z,\beta)}\E^{\alpha^{\star},\beta}\bigg[g\Big(\ell - Y_T^{R,Z,\beta}\Big)\bigg].
$
\end{proposition}

 We report the proof for completeness in  Subsection~\ref{subsec:cara_xi_represent_by_Z_and_y} as we could not find it in the literature.

\subsection{Explicit Characterization of the Principal's Value Function}

Now we consider a quadratic BSDE
\begin{align}\label{eq:bsde_quadratic_Gamma}
    \begin{aligned}
        d   P_t =    Q_t \cdot dX_t    -F_t(   Q_t)dt, \qquad
           P_T = \ell, \qquad \P_0\text{-a.s.}
    \end{aligned}
\end{align}
    where \begin{align}\label{eq:def_F}F_t(   q) &:= \max_{\substack{z\in\R^d \\  b\in B} }\ \max_{\zeta\in\partial_z H_t(z,b)} \tilde{F}_t(q,z,\zeta,b),
    \\
    \tilde{F}_t(   q,z,\zeta,b)&
    :=H_t(z,b)+(q\!-\!z)\!\cdot\!\zeta  -\tfrac{\eta_A}{2}|z|^2 - \tfrac{\eta_P}{2}|   q - z|^2.
    \end{align}
    Unlike in the previous section, here the process $   P$ incorporates information from the Agent's and the Principal's utility functions. Since $H$ is Lipschitz in $z$, we can verify that $\|F_{\cdot}(0)\|_{\infty}<\infty$, and
    \begin{align*}
        |F_{t}(q) - F_{t}(q')|\leq L_F(1+|q|+|q'|)|q-q'|,
    \end{align*}
    for a constant $L_F$ depending only on $\eta_A,\eta_P,\|f\|_{\infty},\|\mu\|_{\infty}$. Details are provided in Appendix~\ref{app:quadratic_bsde}. Hence, by \citet*[Proposition 2.1]{BRIAND20132921} and \citet*{kobylanski_quadra_bsde}, it has a unique solution $(   P,    Q)$ such that $   P\in\S^{\infty}$, and $\big\{\int_0^t   Q_s\cdot dW^{\alpha^{\star},\beta}_s\big\}_{t\in[0,T]}$ is a $\P^{\alpha^{\star},\beta}$-BMO martingale for every $(Z,\beta)\in\H\times \B$ and $\alpha^{\star}\in \hat{\Ac}^{\cara}(Z,\beta)$, i.e. $\|   Q\|_{BMO_2}^{\P^{\alpha^{\star},\beta}}<\infty$, where
$$
\|   Q\|_{BMO_2}^{\P}:=\sup_{\tau\in\T}\bigg\|\E^{\P}\bigg[\int_{\tau}^T |   Q_s|^2ds\bigg|\F_{\tau}\bigg]^{\frac{1}{2}}\bigg\|_{\infty}.
$$ 
\begin{remark}
     Formally, the risk-neutral Agent setting corresponds to $\eta_A = 0$, it is easy to check that $F_t(q) = \bar{H}_t(q)$, and the BSDE~\eqref{eq:bsde_quadratic_Gamma} coincides with the BSDE~\eqref{eq:BSDE_of_Gamma} of the previous section.
\end{remark}

Our main result states, similar to the previous section that the quadratic BSDE~\eqref{eq:bsde_quadratic_Gamma} yields an explicit representation of the Principal’s value function \eqref{eq:cara_principal_value}, together with a characterization of an optimal contract.
\begin{theorem}\label{prop:super_Mg_cara}
 Under the above setting of exponential utilities Principal--Agent problem, we have $
    V^P_{\cara} = -e^{-\eta_P(   P_0 -  R)},$
    with an optimal contract $(\xi^\star,\beta^\star)$, where $
    \xi^{\star} = Y_T^{R,Z^{\star},\beta^{\star}}$, and
    $(Z^{\star},\beta^{\star})$ satisfies \begin{align*}
(Z^{\star}_t,\beta^{\star}_t)\in\argmax_{z\in\R^d, \   b\in B }\ \max_{\zeta\in\partial_z H_t(z,b)} \tilde{F}_t(   Q_t,z,\zeta,b),\qquad \rd\P_0\otimes \rd t-a.e.
\end{align*}   
\end{theorem}
\begin{proof}
    \textit{Step 1.} For every $(y,\beta)\in\R\times\B$, $Z\in\Zc^{\beta}$, define $M^{y,Z,\beta}_t:= -e^{-\eta_P(P_t - Y_t^{y,Z,\beta})}$, $t\in[0,T]$. 
For every $\alpha^{\star}\in\hat{\Ac}^{\cara}(Z,\beta)$, let $\mu^{\star}_t:= \mu_t(\alpha^{\star}_t,\beta_t)\in \partial_z H_t(Z_t,\beta_t)$.  By Itô's formula, \begin{align*}
    dM_t^{y,Z,\beta} &= M_t^{y,Z,\beta}\Big[ - \eta_P d\big(   P_t - Y_t^{y,Z,\beta}\big) + \tfrac{1}{2}\eta_P^2d\big\langle   P - Y^{y,Z,\beta},   P - Y^{y,Z,\beta} \big\rangle_t\Big]\\
    &=\eta_P M_t^{y,Z,\beta}\Big[\big(F_t(   Q_t) - \tilde{F}_t(   Q_t,Z_t,\mu^{\star}_t,\beta_t)\big)dt +( Z_t -    Q_t)\!\cdot\!dW_t^{\alpha^{\star},\beta}\Big].
\end{align*}
Therefore $M^{y,Z,\beta}$ is a $\P^{\alpha^{\star},\beta}$-local supermartingale. Since $Z\in \Zc^{\beta}$ and $   P$ is bounded, we obtain $M^{y,Z,\beta}\in \S$. In particular $\{M^{y,Z,\beta}_{\tau}:\tau\in\T\}$ is $\P^{\alpha^{\star},\beta}$-uniformly integrable. Apply Doob-Meyer decomposition we know $M^{y,Z,\beta}$ is a $\P^{\alpha^{\star},\beta}$-supermartingale. Therefore $\E^{\alpha^{\star},\beta} \big[M_T^{y,Z,\beta}\big]\leq -e^{-\eta_P(P_0 -  y)}.
$
Taking the supremum over $y\geq R$, we obtain $V^P_{\exp}\leq-e^{-\eta_P(   P_0 -  R)}.$

 \medskip

 \textit{Step 2.} Let $F^{\circ}_t(   q,z,b):= \sup_{\zeta\in\partial_z H_t(z,b)}\tilde{F}_t(   q,z,\zeta,b)$. Since $A$ is compact and $(a,z,b)\mapsto h_t(\x,z,a,b)$ is continuous by assumption, by Berge's maximum theorem $(z,b)\mapsto H_t(z,b)$ is continuous, and therefore $(z,\zeta,b)\mapsto \tilde{F}_t(   q,z,\zeta,b)$ is continuous. Notice that $(z,b)\mapsto \partial_zH_t(z,b)$ is upper hemi-continuous, apply Berge's maximum theorem again and we obtain $(z,b)\mapsto F^{\circ}_t(   q,z,b)$ is upper semi-continuous. Notice that $A$, $B$ are compact and $\sup_{b\in B}F^{\circ}_t(   q,z,b) \rightarrow -\infty$ as $|z|\rightarrow+\infty$, so for every $(t,\x,   q)\in[0,T]\times\Omega\times\R^d$, there exists $(\hat{z},\hat{b})$, such that $(\hat{z},\hat{b})\in\argmax_{z\in\R^d, \   b\in B } F^{\circ}_t(   q,z,b)$. Let $A^{\star}_t(z,b):=\argmax_{a\in A} h_t(z,a,b)$, by a standard measurable selection argument, there exists predictable processes $(Z^{\star},\beta^{\star}, \alpha^{\star})$ that satisfies \begin{align*}
(Z^{\star}_t,\beta^{\star}_t)\in\argmax_{z\in\R^d, \   b\in B } F^{\circ}_t(   Q_t,z,b), \qquad \alpha^{\star}_t\in\argmax _{a\in A^{\star}_t(Z^{\star}_t,\beta^{\star}_t)}(Q_t-Z^{\star}_t)\cdot\mu_t(a,\beta^{\star}_t), \qquad \rd\P_0\otimes \rd t-a.e.\end{align*} 
Since $\tilde{F}$ is affine in $\zeta$, we directly obtain $\alpha^{\star}\in\hat{\mathcal{A}}^{\exp}(Z^{\star},\beta^{\star})$ and $F_t(Q_t) = \tilde{F}_t(   Q_t,Z^{\star}_t,\mu^{\star}_t,\beta^{\star}_t)$. We verify that $Z^{\star}\in\Zc^{\beta^{\star}}$. For simplicity, we use $\hat{\P}$ to denote $\P^{\alpha^{\star},\beta^{\star}}$. By definition of $Z^\star$
$$
(   Q_t-Z_t^{\star})\cdot \mu_t(\alpha_t^\star,\beta_t^{\star})+H_t(Z_t^{\star},\beta^\star_t)-\tfrac{\eta_A}{2}|Z_t^{\star}|^2\geq H_t(   Q_t,\beta^{\star}_t)-\tfrac{\eta_A}{2}|   Q_t|^2.
$$
Therefore $
   |Z^{\star}_t|^2\leq C_F(1+|   Q_t|^2)
   $, for some constant $C_F>0$, and thus $\big\{\int_0^tZ^{\star}_s\cdot dW^{\hat{\P}}_s\big\}_{t\in[0,T]}$ is a $\hat{\P}$-BMO martingale. By \citet*[Theorem 2.2]{kazamaki1994continuous}, we have:
   \begin{align*}
       \E^{\hat{\P}}\bigg[\exp \bigg(\varepsilon^2\int_0^t|Z^{\star}_s|^2ds \bigg)\bigg]\leq \frac{1}{1-\varepsilon \|Z^{\star}\|_{BMO_2}^{\hat{\P}}}
       ~~\text{for all}~~
       \varepsilon>0
       ~~\text{with}~~
       \varepsilon\|Z^{\star}\|_{BMO_2}^{\hat{\P}}<1.
   \end{align*}
For every $\eta_A'>\eta_A$, 
   \begin{align*}
     \sup_{t\in[0,T]}  e^{-\eta_A^{\prime} Y_t^{0, Z^{\star},\beta^{\star}}}&=\sup_{t\in[0,T]}\Ec_t\Big(-\eta_A' \int_0^{\cdot} Z^{\star}_s
       \!\cdot\!d W^{\hat{\P}}\Big) \exp \bigg(\int_0^t\big(\tfrac{\eta_A'(\eta_A'-\eta_A)}{2}|Z^{\star}_s|^2-\eta_A' f_s(\alpha^{\star}_s,\beta^{\star}_s)\big) d s\bigg)\\
       &\leq C_{\eta_A'} \sup_{t\in[0,T]}\Ec_t\bigg(-\eta_A' \int_0^{\cdot} Z^{\star}_s
       \!\cdot\!d W^{\hat{\P}}\bigg) \exp \bigg(C_{\eta_A,\eta_A'} \int_0^T|Z^{\star}_s|^2 ds\bigg),
   \end{align*}
   where $C_{\eta_A'}  = e^{\eta_A'T\|f\|_{\infty}}$, $C_{\eta_A,\eta_A'} = \tfrac{\eta_A'(\eta_A'-\eta_A)}{2}$.
 By \citet*[Theorem 3.1]{kazamaki1994continuous}, we can choose $p >1$ sufficiently close to $1$ such that for every $\eta_A'\in(\eta_A,2\eta_A)$, $\Ec_T\big(-\eta_A'\int_0^{\cdot}Z^{\star}_s\!\cdot\!dW^{\hat{\P}}_s\big)\in\L^{p}(\hat{\P})$. Let $q$ satisfy $\tfrac{1}{p}+\tfrac{1}{q} =1$, by Hölder inequality, we have 
   \begin{align*}
      \E^{\hat{\P}} \bigg[\!\sup_{t\in[0,T]}e^{-\eta_A^{\prime} Y_t^{0, Z^{\star},\beta^{\star}}}\bigg]&\!\leq\! C_{\eta_A'}\E^{\hat{\P}}\bigg[\!\sup_{t\in[0,T]}\mathcal{E}_t^p\bigg(\!-\eta_A'\! \int_0^{\cdot} \!Z^{\star}_s
      \!\cdot\!d W^{\hat{\P}}_s\bigg) \bigg]^{\frac{1}{p}} \E^{\hat{\P}}\bigg[\!\exp \bigg(qC_{\eta_A,\eta_A'}\!\int_0^T\!|Z^{\star}_s|^2\! d s\bigg)\bigg]^{\frac{1}{q}}.
   \end{align*}
By Doob's inequality, we know $\E^{\hat{\P}}\big[\sup_{t\in[0,T]}\mathcal{E}_t^p\big(-\eta_A' \int_0^{\cdot} Z^{\star}_s
      \!\cdot\!d W^{\hat{\P}}_s\big) \big]<\infty$. Now choose $\eta_A'$ such that $\tfrac{q\eta_A'(\eta_A' - \eta_A)}{2}<\varepsilon^2$, we deduce that $  \E^{\hat{\P}} \big[\sup_{t\in[0,T]}e^{-\eta_A^{\prime} Y_t^{0, Z^{\star},\beta^{\star}}}\big]<\infty$.  On the other hand, as $
   d   P_t =    Q_t\!\cdot\! dX_t   -\tilde{F}_t(   Q_t,Z^{\star}_t,{\mu}^{\star}_t,\beta^{\star}_t)dt
   $, we obtain 
   $$
   Y_T^{0,Z^{\star},\beta^{\star}} = \ell-   P_0 +\int_0^T(Z^{\star}_t-   Q_t)\!\cdot\!dW^{\hat{\P}}_t - \frac{\eta_P}{2}\int_0^T|Z^{\star}_t  -    Q_t|^2dt.
   $$
   Therefore $e^{\eta_P^{\prime} Y_T^{0, Z^{\star},\beta^{\star}}}=e^{\eta_P'(\ell -    P_0)}\Ec_T\big(\eta_P' \int_0^{\cdot} (Z^{\star}_s-   Q_s)\!\cdot\!d W^{\hat{\P}}\big)  e^{\tfrac{\eta_P'(\eta_P'-\eta_P)}{2}\int_0^T|Z^{\star}_s -    Q_s|^2d s}.
$ where, similar to the previous proof, $\big\{\int_0^t(Z^{\star}_s-   Q_s)\cdot dW^{\hat{\P}}_s\big\}_{t\in[0,T]}$ is also a $\hat{\P}$-BMO martingale. Then it follows from Hölder's inequality that
$\E^{\hat{\P}} \big[e^{\eta_P^{\prime} Y_T^{0, Z^{\star},\beta^{\star}}}\big]<\infty$, for some $\eta_P'>\eta_P$. Combine with Remark~\ref{rmk:Z_beta_equal_bar_Z_beta}, we conclude that $Z^{\star}\in\Zc^{\beta^{\star}}$.

    \medskip

 \textit{Step 3.} Notice that $
    dM_t^{y,Z^{\star},\beta^{\star}}  =\eta_P M_t^{y,Z^{\star},\beta^{\star}}( Z^{\star}_t -    Q_t)\cdot dW_t^{\alpha^{\star},\beta^{\star}}$, so $M^{y,Z^{\star},\beta^{\star}}$ is a $\P^{\alpha^{\star},\beta^{\star}}$-uniformly integrable martingale, which implies $$\E^{\alpha^{\star},\beta^{\star}}\big[g(\ell - Y_T^{R,Z^{\star},\beta^{\star}})\big]=-e^{-\eta_P(   P_0 -  R)}. $$
Combine with previous steps we complete the  proof.
\end{proof}

\subsection{Proof of Proposition \ref{prop:reduce_to_standard_control_cara}}\label{subsec:cara_xi_represent_by_Z_and_y}

Obviously, for every $(y,\beta)\in\R\times\B$ and $Z\in\Zc^{\beta}$, we have $Y_T^{y,Z,\beta}\in \Xi_{\exp}$, and  $V^A_{\cara}\big(Y^{y,Z,\beta}_T,\beta\big) = g_A(y),$ by a simple verification argument. In order to complete the proof of Proposition~\ref{prop:reduce_to_standard_control_cara}, we now prove: for  every $(\xi,\beta)\in \Xi_{\exp}\times \B$, there exists a unique $(y,Z)\in\R\times\Zc^{\beta}$, s.t.  $\xi = Y_T^{y,Z,\beta}$.\\

 Fix $(\xi,\beta)\in\Xi_{\exp}\times \B$, without loss of generality we let $r = 0$. Consider the BSDE
\begin{align}\label{eq:backward_SDE_of_Y_exponential}
    dY_t = Z_t\cdot dX_t + \tfrac{\eta_A}{2}
    |Z_t|^2dt - H_t(Z_t, \beta_t)dt,\qquad Y_T = \xi,
\end{align}
 If the above BSDE has a solution, we know $Y_0\in\R$ because the filtration $\F_0$ is trivial. So it suffices to prove the above BSDE has a unique solution $(Y,Z)$ such that $Z\in\Zc^{\beta}$.
 Consider the BSDE :\begin{align}\label{eq:BSDE_of_tilde_Y}
     d\tilde{Y}_t = \tilde{Z}_t\cdot dX_t+\tilde{H}_t(\tilde{Y}_t,\tilde{Z}_t,\beta_t)dt,\qquad \tilde{Y}_T = \tilde{\xi},
 \end{align}
 where $\tilde{H}_t(y,z,b) := \sup_{a\in A} \tilde{h}_t(y,z,a,b)$, $\tilde{h}_t(y,z,a,b):=-\mu_t(a,b)\cdot z - \eta_A y f_t(a,b)$, and $\tilde{\xi} = e^{-\eta_A \xi}$. $\xi\in\Xi_{\cara}$ implies that $\tilde{\xi}\in\L$. Since $\mu,f$ are bounded, $\tilde{H}$ is Lipschitz in $(y,z)$. Therefore, by $L^p$-BSDE theory \cite[Theorem 4.2]{BRIAND2003109}, we know the Equation \eqref{eq:BSDE_of_tilde_Y} has a unique solution $(\tilde{Y}, \tilde{Z})\in \S\times \H$. Since $A$ is compact and $a\mapsto \tilde{h}_t(y,z,a,b)$ is continuous, $\argmax_{a\in A}\tilde{h}_t(y,z,a,b)\neq \emptyset$ for every $(y,z,b)\in\R\times\R^d\times B$. By a standard measurable selection argument, we are able to select an optimizer $\tilde{\alpha}_t\in \argmax_{a\in A}\tilde{h}_t(\tilde{Y}_t,\tilde{Z}_t,a,\beta_t)$, which is progressively measurable.
By Itô's formula, we derive $$
 d\big(\tilde{Y}_t e^{\eta_A \int_0^t f_s(\tilde{\alpha}_s,\beta_s)ds}\big) = e^{\eta_A \int_0^t f_s(\tilde{\alpha}_s,\beta_s)ds}\tilde{Z}_t\cdot dW_t^{\tilde{\alpha},\beta},
 $$
 Since $\tilde{Y}\in\S$, $f$ is bounded, we know $\big\{\tilde{Y}_t e^{\eta_A \int_0^t f_s(\tilde{\alpha}_s,\beta_s)ds}\big\}_{t\in[0,T]}$ is a $\P^{\tilde{\alpha},\beta}$-martingale, therefore 
 \begin{align}\label{eq:explict_Mg_tilde_Y}
 \tilde{Y}_t= \E^{\tilde{\alpha},\beta}\Big[e^{-\eta_A\xi+\eta_A \int_t^T f_s(\tilde{\alpha}_s,\beta_s)ds}\Big|\F_t\Big].
  \end{align}
 Notice that $\tilde{\alpha}$ actually depends on $\tilde{Y}$ and $\tilde{Z}$, this provides an implicit formula of $\tilde{Y}$. Additionally we know $\tilde{Y}_t>0$ almost surely. Now we are able to define $Y_t: = -\tfrac{1}{\eta_A}\log \tilde{Y}_t$ and $Z_t: = -\tfrac{\tilde{Z}_t}{\eta_A \tilde{Y}_t}$.
 The uniqueness of $Z$ is inherited from that of the solution to \eqref{eq:BSDE_of_tilde_Y}. Now we verify $Z\in \Zc^{\beta}$:
 \begin{enumerate}
     \item Since $\tilde{Y}>0$, by taking $ \tau_n:=\inf \big\{t \geq 0: \tilde{Y}_t \leq \tfrac{1} { n}\big\}$, we obtain that $Z \in \H_{\loc}$.
     \item $\tilde{Y}\in\S$ is equivalent to $\E^{\P_0}\Big[\sup_{t\in [0,T]}e^{-\eta_A' Y_t}\Big]<\infty $ for some $\eta_A'>\eta_A$. 
     \item By Remark \ref{rem:Lq_equal_Lp}, there exists $\eta_P'>\eta_P$ such that $\E^{\tilde{\alpha},\beta}\big[e^{\eta_P' \xi}\big]<\infty$. Define $\varphi(x) := x^{-\frac{\eta_P'}{\eta_A}}$, which is a convex function on $(0,+\infty)$. Applying Jensen's inequality to Equation \eqref{eq:explict_Mg_tilde_Y},  we derive:\begin{align}\label{eq:eq_for_remark4.3}
         e^{\eta_P'Y_t } = \varphi(\tilde{Y}_t )\leq  \E^{\tilde{\alpha},\beta}\Big[\varphi\Big(e^{-\eta_A\xi+\eta_A \int_t^T f_s(\tilde{\alpha}_s,\beta_s)ds}\Big)\Big|\F_t\Big]\leq \E^{\tilde{\alpha},\beta}\big[e^{\eta_P' \xi}|\F_t\big] e^{\eta_P' T \|f\|_{\infty}}.
     \end{align}
     By Doob’s inequality, we conclude that $\E^{\tilde{\alpha},\beta}\Big[\sup_{t\in [0,T]}e^{\eta_P'' Y_t}\Big]<\infty $ for some $\eta_P''>\eta_P$. \qed
 \end{enumerate}

\begin{remark}\label{rmk:Z_beta_equal_bar_Z_beta}
    A direct corollary of Equation~\eqref{eq:eq_for_remark4.3} is that $\Zc^{\beta}= \bar{\Zc}^{\beta}$.
\end{remark}

\section{Principal delegating to a Regime-Switching Agent}\label{sec:PA_RS}

In this section, we introduce a Principal–Agent problem in which the Agent solves a regime-switching problem, and we derive an explicit solution using the arguments developed in the previous section. The Sannikov's method raises serious difficulties in the present context. The Agent's value function can still be represented by a  reflected system of BSDEs, but unlike the classical Principal-Agent problem of Section~\ref{sec:PA_problem}, reversing the dynamics to the forward direction does not reduce the problem to a standard control problem.

\subsection{Principal-Agent model with regime switching output}
Given $m\in\N$, $m\geq 2$, let $\M := \set{0,1,\cdots ,m-1}$ denote the set of all the regimes, between which the agent may switch. We use $\Es$ to denote the set of all the marked point processes $E:[0,T]\times \Omega\longrightarrow \M $ start at regime $0$, i.e.: 
    \begin{align*}
\Es
:= \Big\{E \in \mathbb{D}([0,T];\M): 
& ~~E_t=\sum _{n\geq 0}I_n \mathds{1}_{[\tau_n,\tau_{n+1})} (t), \quad 0\leq t<T,\quad  E_T = E_{T-} ,\\
& ~~\text{for some non-decreasing}~(\tau_n)_{n\ge 1}\subset\T,
~\tau_0=0,~\tau_n\uparrow T,
\\ & ~~ \text{and some}~
   (I_n)_{n\ge 0}\subset\L^0(\F_{\tau_n},\M),
   ~I_{n-1}\neq I_{n}\text{ on }\{\tau_n<T\} \Big\}. 
\end{align*}
We use the convention  $E_{0-}=I_{-1}=0$ for every $E\in\Es$, and we allow an initial switch at $0$. Here $E_t$ denotes the regime of the Agent at time $t$. Let $\M_0 = \set{1,\cdots, m-1}$. For every $E\in\Es$, we define the associated integer-valued optional random measure $\rho^E$ on $[0,T]\times\M_0$ and counting process $N$: 
$$
\rho^E(\omega;dt,di)
:=
\sum_{\Delta E_t(\omega)\neq 0}
\delta_{\left(t,\Delta E_t(\omega)\right)}(dt,di),
~~\text{and}~~
N^i_t := \rho^E([0,t],\{i\}),
$$
with $\Delta E_t=E_t-E_{t-}~(\text{mod } m)$ as standard. Then, 
$$
dE_t = \int_{\M_0} i \rho^E(dt, di),\quad(\text{mod } m)
~~
\text{or equivalently,}~~
dE_t = \sum_{i\in\M_0} i dN^i_t,\quad (\text{mod } m),
$$
as $\M_0$ is discrete. Notice that $N^i_0$ counts the possible initial jump at $0$, so it is not necessarily 0. Let $N_t:= \sum_{i\in\M_0}N^i_t$ be the total number of jumps.\\

We extend the coefficients $\mu$, $f$ of Section \ref{sec:PA_problem} allowing an additional dependence on $i\in\M$:
$$(\mu,f): [0,T]\times\Omega\times\M\times A\times B\longrightarrow \R^d\times \R, $$
and we assume that, for every $i\in\M$, $\mu^i$, $f^i$ satisfy the same assumptions as in Section~\ref{sec:PA_problem}, i.e. $\mu^i\in\mathcal{D}^d \cap \L^{\infty}$; $f^i\in \mathcal{D}^1$ and satisfies the integrability assumption~\eqref{eq:f_integ_assump}.
Similarly, for every $(E,\alpha,\beta)\in\Es\times\A\times\B$, there exists a unique probability measure $\P^{E,\alpha,\beta}$, s.t.\begin{equation*}
    dX_t = \mu^{E_t}_t(X,\alpha_t,\beta_t)dt+ dW^{E,\alpha,\beta}_t,
\end{equation*}
for some $\P^{E,\alpha,\beta}$-Brownian motion $W^{E,\alpha,\beta}$.
We introduce a switching cost rate process $$c:[0,T]\times\Omega\times \M^2\longrightarrow\R_+,$$ which represents the instant switching cost from one regime to the other. We assume for every $i,j
\in\M$, $c^{i,j}$ is predictable and continuous in $t$, and 
\begin{align*}
c^{i,i}_t = 0,~~
c^{i,j}_t+c^{j,k}_t > c^{i,k}_t,
& ~~\text{ for all}~i,j,k\in \M,\ i\neq j,\ j\neq k;
\\
\inf_{\substack{i,j\in\M \\ i\neq j}}
\inf_{t\in[0,T]} c_t^{i,j} \ge c_b,
&~~\text{ for some constant } c_b>0.
\end{align*}

Recall the discount factor is defined by $\Rc_t = e^{-\int_0^t r(s)ds}$, with a bounded deterministic function $r$. For every $
(\xi,\beta)\in\L\times\B$, we denote the objective function of the Agent by: \begin{align}
    J^A_{RS}(\xi,\beta;E,\alpha) &:=\E^{E,\alpha,\beta}\bigg[\Rc_T\xi-\int_0^T\Rc_tf_t^{E_t}(\alpha_t,\beta_t)dt 
- \sum_{i=1}^{m-1}\int_{[0,T]} \Rc_tc^{E_{t-},E_t}_tdN^{i}_{t} \bigg],
\end{align}
Therefore the agent's value function is given by
\begin{equation}
    V^A_{RS}(\xi,\beta) :=  \sup_{E\in\Es,\alpha\in\A}J^A_{RS}(\xi,\beta;E,\alpha).
\end{equation}
\begin{remark}
Some works on regime switching problem use left-continuous regime processes with right limits, which can simplify verification argument, see e.g. \citet*{hamadene_on_the_starting_and_stopping_problem}. We instead consider the càdlàg regime processes $E$, which is more natural for our marked point process formulation. The two formulations are characterized by the same reflected BSDE and therefore provide the same value function.
\end{remark}
We denote the set of all admissible contracts  by: \begin{equation}\label{eq:admissable_contract_regime_switching}
\C_{RS}:=\set{ (\xi,\beta)\in \L\times\B: V^A_{RS}(\xi,\beta)\geq R}.
\end{equation}
In the next section, we prove the existence of an optimizer $\big({E}^{\star}(\xi,\beta),{\alpha}^{\star}(\xi,\beta)\big)\in\Es\times\A$, see \eqref{eq:optimal_E_alpha}, and we will omit the index $(\xi,\beta)$ for simplicity. Then we define the Principal’s problem:
\begin{align}\label{eq:principal_value_reflect}
V^P_{RS} := \sup_{ (\xi,\beta)\in\C_{RS}} \E^{{E}^{\star},{\alpha}^{\star},\beta}\big[g(\ell-\xi)\big].
\end{align}
\begin{remark}
Unlike the standard Principal-Agent setup, we only consider one particular optimal response of the Agent, rather than the whole set of his optimal responses, mainly due to the intractability of all the optimizers in the present context. 
\end{remark}

\subsection{Solution to the Agent’s Problem}

Define the Hamiltonian $H:[0,T]\times \M\times\Omega\times \R^d\times B\longrightarrow \R$ by:
$$
H^i_t(z,b): =\sup_{a\in A} h_t^i(z,a,b), 
 \qquad h^i_t(z,a,b): = z\cdot \mu_t^i(a,b) - f^i_t(a,b).
$$

 Given $(\xi,\beta)\in \L\times\B$, consider the following system of reflected BSDE: for $i\in\M$,
\begin{align}
\label{eq:skorokhod_condition_K}
\left\{
\begin{aligned}
  &   Y^i_t =\xi+ K^i_T-K^i_t-\int_t^Tr(s)Y^i_sds+\int_t^TH^i_s(Z^i_s,\beta_s)ds-\int_t^TZ^i _sdX_s, \\
   & Y^i_t \geq \max_{j\neq i}\set{Y^j_t - c^{i,j}_t},
     ~~\text{and}~~
    (Y^i_t -\max_{j\neq i}\set{Y^j_t - c^{i,j}_t})dK^i_t = 0,   
\end{aligned}
\right.
\end{align}

which has Lipschitz coefficients. Under our assumption on $c^{i,j}$, the
$L^p$-well-posedness theory for systems of reflected BSDEs (see \citet*{Shi2024}) ensures this reflected BSDE has a unique strong solution $(Y,Z,K)\in\S_c\times\H\times \I_c$.

\begin{proposition}\label{prop:solve_agent_problem_regime} For every $(\xi,\beta)\in\L\times\B$, we have $V^{A}_{RS}(\xi,\beta) = Y^{0}_0 ,$ with an optimizer given by \begin{align}\label{eq:optimal_E_alpha}
         E^{\star}_t := \sum_{i\geq 0} \hat{\cI}_i \mathds{1}_{[\hat{\tau}_{i},\hat{\tau}_{i+1})}(t),~~~~
          \alpha^{\star}_t := \hat{\alpha}^{E^{\star}_t}_t\big(Z^{E^{\star}_t}_t,\beta_t\big),
          ~~\text{where:}
     \end{align}
\begin{itemize}
    \item $\hat{\alpha}^{i}_t(z,b) \in \arg\max_{a\in A} h_t^i(z,a,b)$;
    \item $\hat{\cI}_0 \in \arg \max _{j \in \mathcal{M}}\Big\{Y_0^j-c_0^{0, j}\Big\}$, and $\hat{\cI}_0:=0 $ whenever $ 0\in \arg \max _{j \in \mathcal{M}}\Big\{Y_0^j-c_0^{0, j}\Big\}$;
    \item for $i\geq 1$, $\hat{\tau}_{i} := \inf\big\{t>\hat{\tau}_{i-1}: Y_t^{\hat{\cI}_{i-1}} = \max_{j\neq \hat{\cI}_{i-1}}\{Y_t^{j}-c_t^{\hat{\cI}_{i-1},j} \}\big\}$,
    
    \hspace{15mm}
    $\hat{\cI}_i \in  \big\{j\in\M\setminus \{\hat{\cI}_{i-1}\}:Y_{\hat{\tau}_{i}}^{\hat{\cI}_{i-1}}- Y_{\hat{\tau}_{i}}^{j}+c_{\hat{\tau}_{i}}^{\hat{\cI}_{i-1},j} = 0\big\}.
         $
\end{itemize} 
\end{proposition}

\begin{proof}
 The result follows the standard verification argument in stochastic control theory and is an easy extension of \citet*{hamadene_on_the_starting_and_stopping_problem} to the present multiple regime case. We provide the details for completeness. For every $ (E,\alpha,\beta)\in\Es\times\A\times\B$, let $\tau_0 = 0$, $\tau_i := \inf{\set{t>\tau_{i-1}: \Delta E_t\neq 0 }}$ be the i-th jump time, and $\cI_i = E_{\tau_i}$ be the corresponding regime. \begin{align*}
        Y^0_0
        \geq Y_0^{E_0} - c_0^{0,E_0}
        &\geq \Rc_{\tau_1}Y^{E_0}_{\tau_1  }-c_0^{0,E_0}-\int_0^{\tau_1  }\Rc_tf_t^{E_0}(\alpha_t,\beta_t)dt - \int^{\tau_1 }_0\Rc_tZ_t^{E_0}\big(dX_t-\mu^{E_0}_t(\alpha_t,\beta_t)dt\big)\\
        &\geq \Rc_{\tau_1}\xi\ch{\tau_1 = T}-c_0^{0,\cI_0}+\Rc_{\tau_1}(Y^{\cI_1}_{\tau_1 }-c^{\cI_0,\cI_1}_{\tau_1})\ch{\tau_1<T}-\int_0^{\tau_1 }\Rc_{t}f_t^{E_0}(\alpha_t,\beta_t)dt  \\
        &\qquad-\int^{\tau_1 }_0\Rc_{t}Z_t^{E_0}\big(dX_t-\mu^{E_0}_t(\alpha_t,\beta_t)dt\big).
        \end{align*}
        Therefore, by induction, we can derive, for every $n$:\begin{align*}
        Y^0_0
        &\geq \Rc_{\tau_n}\xi\ch{\tau_n = T}+\Rc_{\tau_n}Y^{E_{\tau_n}}_{\tau_n }\ch{\tau_n<T}-\sum_{j=0}^n\Rc_{\tau_j}c_{\tau_j}^{\cI_{j-1},\cI_j}\ch{\tau_j<T}-\int_0^{\tau_n }\Rc_{t}f_t^{E_t}(\alpha_t,\beta_t)dt \\&\qquad - \int^{\tau_n }_0\Rc_{t}Z_t^{E_{t}}\big(dX_t-\mu^{E_t}_t(\alpha_t,\beta_t)dt\big).
        \end{align*}
    Since $c^{i,j}$ are uniformly bounded from below, we assume without loss of generality that $\inf\{n:\tau_n = T\}<\infty$, a.s. because otherwise $J^A_{RS}(\xi,\beta;E,\alpha)  = -\infty$. Letting $n\rightarrow\infty$, we have 
    \begin{align*}
        Y^0_0&\geq \Rc_{T} \xi - \sum_{j=0}^\infty \Rc_{\tau_j}c_{\tau_j}^{\cI_{j-1},\cI_j}\ch{\tau_j<T} -\int_0^{ T}\Rc_{t}f_t^{E_t}(\alpha_t,\beta_t)dt - \int^{ T}_0\Rc_{t}Z_t^{E_{t}}dW^{E,\alpha,\beta}_t.
        \end{align*}
    Notice that $\big|Z_s^{E_s}\big|\leq \sum_{i = 0}^{m-1}|Z^i_s|$, and $Z^i\in \H$, so $\big(Z^{E_s}_s\big)_{s\in[0,T]}\in\H$, i.e. $\big(Z^{E_s}_s\big)_{s\in[0,T]}\in\H^p$ for some $p>1$. Then by BDG inequality, $$\E^{E,\alpha,\beta}\bigg[\Big|\sup_{t\in[0,T]}\int^{ t}_0Z_s^{E_{s}}\cdot dW^{E,\alpha,\beta}_s\Big|^p\bigg]\leq C_p\E^{E,\alpha,\beta}\bigg[\bigg(\int^{ T}_0\big|Z_s^{E_{s}}\big|^2ds\bigg)^{\frac{p}{2}}\bigg]<\infty,
    $$
  for a constant $C_p$. Therefore, $\int^{ t}_0\Rc_sZ_s^{E_{s}}dW^{E,\alpha,\beta}_s$ is a $\P^{E,\alpha,\beta}$-martingale, and we have
    \begin{align*}
        Y^0_0&\geq \E^{\P^{E,\alpha,\beta}}\bigg[\Rc_{T}\xi - \int_0^{T}\Rc_{t}f_t^{E_t}(\alpha_t,\beta_t)dt -\sum_{j=0}^\infty \Rc_{\tau_j}c_{\tau_j}^{\cI_{j-1},\cI_j}\ch{\tau_j<T}\bigg].
    \end{align*}
Taking the supremum over all $(E,\alpha)\in\mathscr{E}\times\A$ yields
\begin{equation}\label{eq:y0_leq_VARS}
Y^{0}_{0}\;\ge\;V^{A}_{RS}(\xi,\beta).
\end{equation}
Since $A$ is compact and $a\mapsto h^i_t(z,a,b)$ is continuous, $\arg\max_{a\in A} h_t^i(\x,z,a,b)\neq \emptyset$. Therefore we can select $\hat{\alpha}$ measurable such that $\hat{\alpha}^{i}_t(z,b) \in \arg\max_{a\in A} h_t^i(z,a,b)$. It is easy to check that the equality in Equation~\eqref{eq:y0_leq_VARS} is attained at $(E^{\star},\alpha^{\star})$. 
\end{proof}

\subsection{Failure of Sannikov's reduction}
In this section we consider a special case where $m = 2$, i.e. $\M = \set{0,1}$, and we set the discount rate $r\equiv 0$ to shorten the expressions. In this case, the reflected BSDE~\eqref{eq:skorokhod_condition_K} reduces to: 
for $i\in\M$,
\begin{align}
\label{eq:rbsde_2regimes}
\left\{
\begin{aligned}
  &   Y^i_t =\xi+ K^i_T-K^i_t+\int_t^TH^i_s(Z^i_s,\beta_s)ds-\int_t^TZ^i _sdX_s, \\
   & Y^i_t \geq Y^{1-i}_t - c^{i,1-i}_t,
     ~~\text{and}~~
     (Y^i_t -Y^{1-i}_t +c^{i,1-i}_t)dK^i_t = 0.
\end{aligned}
\right.
\end{align}
Two obstacles prevent a direct reduction of the Principal's value function to a standard control problem. First, we have the state constraint that $\xi$ does not depend on the regime, i.e. $Y^0_T = Y^1_T$; second, the reflection introduces the additional increasing terms $K^i$, so we can not simply reparametrize $\xi$ by $(Y_0^i,Z^i)$ as in the previous sections. However, since the discount rate $r$ is uncontrolled, we can solve this by considering a doubly reflected BSDE with $0$ terminal value. More precisely, given every $(Z^0,\beta)\in\H\times\B$, consider: \begin{align}
\label{eq:double_reflected_BSDE}
\left\{
\begin{aligned}
  &  \delta Y_t =  K^1_T-K^1_t- K^0_T+K^0_t + \int_t^T\big[H^1_s(\delta Z_s + Z^0_s,\beta_s) -H^0_s( Z^0_s,\beta_s)\big]ds
  -\int_t^T\delta Z _sdX_s, \\
   & - c^{1,0}_t\leq \delta Y_t \leq c^{0,1}_t ,
    ~~( \delta Y_t + c^{1,0}_t)dK^1_t = 0,
    ~~\text{and}~~
    (-\delta Y_t+c_t^{0, 1})d K_t^0=0.
\end{aligned}
\right.
\end{align}
By \citet*{El_Asri_Lp_double_Reflected_BSDE}, it has a unique solution $(\delta Y, \delta Z, K^0,K^1)\in \S_c\times\H\times \I_c\times\I_c$.  Therefore, for every $(y,Z^0,\beta)\in\R\times\H\times \B$, we can define $(Y^0,Y^1,Z^1)$ by \begin{align*}
    Y^{0,Z^0}_t&:=  y - K^{0,Z^0}_t -\int^t_0  H^0_s(Z^0_s,\beta_s)ds + \int^t_0 Z^0_sdX_s,\\
  Y^{1,Z^0}_t  &:=  Y^0_t +\delta Y^{Z^0}_t,\qquad Z^{1,Z^0}_t  :=  Z^0_t +\delta Z^{Z^0}_t.
\end{align*}
Since $\delta Y_T = 0$, we obtain $Y^1_T = Y^0_T$. By the analysis above, we have the following reduction.
\begin{proposition}
With the notation $(E^{\star},\alpha^{\star}) = (E^{\star},\alpha^{\star})\brac[2]{Y_T^{0,(y,Z^0,\beta)}}$ in Proposition~\ref{prop:solve_agent_problem_regime}, we have 
    $$V^P_{RS} = \sup_{y\geq R}\sup_{\beta\in\B}\sup_{Z^0\in\H} \E^{{E}^{\star},\alpha^{\star},\beta}\bigg[g\big(\ell - Y_T^{0,(y,Z^0,\beta)}\Big)\bigg].$$
\end{proposition}

This turns out to be a stochastic control problem of coupled forward-backward SDE (FBSDE). There are only few studies on it; see for instance \citet*{doi:10.1137/090763287}, \citet*{WU20131473} and \citet*{hernandez_dylan_emma_stackelberg_game}. However, to the best of our knowledge, a value function characterization for this stochastic control problem of FBSDE is not available in the existing literature.

\subsection{BSDE characterization of the Principal's problem}
Consider the system of reflected BSDE for every $i\in\M$:
\begin{align}\label{eq:rbsde_p_beta}
\begin{cases}
     \bar{P}^{i}_t = \ell+ \bar{L}^i_T- \bar{L}^i_t - \int_t^Tr(s) \bar{   P}^i_sds+\int_t^T \bar{H}^i_s( \bar{Q}^i_s)ds - \int_t^T \bar{   Q}^{i}_sdX_s,\\
      \bar{   P}^i_t \geq \max_{j\neq i}\set{  \bar{   P}^{j}_t-c^{i,j}_t},
    ~~\text{and}~~
    \big(  \bar{   P}^i_t -\max_{j\neq i}\set{  \bar{   P}^{j}_t-c^{i,j}_t}\big)d \bar{L}^i_t = 0,
\end{cases}
\end{align}

where $\bar{H}:[0,T]\times\M\times \Omega\times\R^d \longrightarrow \R$ is defined as: $$\bar{H}^i_t(\x,z):=\sup_{(a,b)\in A\times B}\big(\mu^i_t(\x,a,b)\cdot z - f^i_t(\x,a,b)\big) = \sup_{b\in B}H^i_t(\x,z,b).$$ Since $\ell\in \cap_{p>1}\L^p$, by \citet*[Theorem 3.3 and Remark 2.1]{Shi2024}, this system of reflected BSDE has a unique solution $(\bar{   P},\bar{   Q},\bar{L})$, s.t. $\bar{   P}^i\in\cap_{p\in(1,p_f)}\S^p_c $, $\bar{   Q}^i\in\cap_{p\in(1,p_f)}\H^p$, $\bar{L}^i\in\cap_{p\in(1,p_f)}\I^p_c$, for some $p_f\in(1,2]$.

\begin{theorem}\label{thm:solve_PA_RS}
    Let $g$ be non-decreasing and continuous with $g^{\conc}<\infty$, then the value function \eqref{eq:principal_value_reflect} is given by $V^P_{RS}  =  g^{\conc}\Big((\bar{   P}^0_0 - R)e^{\int_0^Tr(s)ds}\Big).$ Moreover, if $(\bar{P}^0_0-R)e^{\int_0^Tr(s)ds}\in \big\{g=g^{\conc}\big\}$, then an optimal contract $(\xi^{\star},\beta^{\star})$ is given by:
    $$\xi^{\star} =\ell+ (R - \bar{   P}^0_0)e^{\int_0^Tr(s)ds},
    ~~\beta^{\star}_t := \hat{b}^{\bar{E}^{\star}_t}_t\big(\bar{   Q}^{\bar{E}^{\star}_t}_t\big),
    ~~\text{where:}
    $$
    \begin{itemize}
        \item $\hat{b}^i_t(z) \in \arg\max_{b\in B} H_t^i(z,b)$,
        \item $\bar{E}^{\star}_t := \sum_{i\geq 0} \bar{\cI}_i \mathds{1}_{[\bar{\tau}_{i},\bar{\tau}_{i+1})}(t)$, with $\bar{\tau}_{i} = \inf\big\{t>\bar{\tau}_{i-1}: {   \bar{P}_t^{\bar{\cI}_{i-1}} = \max_{j\neq \bar{\cI}_{i-1}}\{   \bar{P}_t^{j}-c_t^{\bar{\cI}_{i-1},j} }\}\big\}$,

        \hspace{48mm}
        and $\bar{\cI}_i = \min \big\{j\in\M\setminus \{\bar{\cI}_{i-1}\}:   \bar{P}_{\bar{\tau}_{i}}^{\bar{\cI}_{i-1}}-    \bar{P}_{\bar{\tau}_{i}}^{j}+c_{\bar{\tau}_{i}}^{\bar{\cI}_{i-1},j} = 0\big\}$.
    \end{itemize}    
\end{theorem}

\begin{proof}
For all $\beta\in\B$, there is a unique solution $(P,Q,L)$ to the system of reflected BSDEs:
\begin{align}\label{eq:rbsde_p}
\begin{cases}
        P^{i}_t = \ell+L^i_T-L^i_t - \int_t^Tr(s)   P^i_sds+\int_t^TH^i_s(   Q^i_s,\beta_s)ds - \int_t^T   Q^{i}_sdX_s,\\
        P^i_t \geq \max_{j\neq i}\set{    P^{j}_t-c^{i,j}_t},
    ~~\text{and}~~
    \big(P^i_t -\max_{j\neq i}\set{    P^{j}_t-c^{i,j}_t}\big)dL^i_t = 0,
    ~~i\in\M,
\end{cases}
\end{align}
such that $   P^i\in\cap_{p\in(1,p'_f)}\S^p_c $, $   Q^i\in\cap_{p\in(1,p'_f)}\H^p$, $L^i\in\cap_{p\in(1,p'_f)}\I^p$, for some $p_f'\in(1,2]$. Let
\begin{align*}
V^P_{\beta} := \sup_{ \xi:(\xi,\beta)\in\C_{RS}} \E^{{E}^{\star},{\alpha}^{\star},\beta}\big[g(\ell-\xi)\big].
\end{align*}

\textit{Step 1.} We first prove $V^P_{\beta}\leq g^{\conc}\Big(\big(   P^{\beta,0}_0-R\big)e^{\int_0^Tr(s)ds}\Big)$. For every $\xi\in\L$, let $(Y^{\xi},Z^{\xi},K^{\xi})$ denote the solution to the reflected BSDE \eqref{eq:skorokhod_condition_K}. For simplicity, we omit the index $(\xi,\beta)$ and simply write ${E}^{\star} $ and ${\alpha}^{\star} $ when it does not cause ambiguity. By definition of $E^{\star}$, we have $Y^{E^{\star}_t }_t - Y_t^{E^{\star}_{t-}} = c_t^{E^{\star}_{t-}, E^{\star}_t }$ on $\{\Delta N_t^{\star}\neq 0\}$, and $Y^{E^{\star}_t }_t > \max_{i\neq E^{\star}_t}\big\{Y^i_t - c_t^{E^{\star}_{t},i }\big\}$ on $\{\Delta N_t^{\star}= 0\}$. It follows from the Skorokhod condition that $dK_t^{E_{t-}^{\star}} = 0$. Therefore,
 \begin{align*}
d\big(Y_t^{\xi;E^{\star}_t}e^{\int_t^Tr(s)ds}\big) &=e^{\int_t^Tr(s)ds} \Big(f_t^{E^{\star}_t} (\alpha^{\star}_t,\beta_t)dt+ Z_t^{E^{\star}_t} dW_t^{E^{\star},\alpha^{\star} , \beta} +c_t^{E^{\star}_{t-}, E^{\star}_t }dN^{\star}_t\Big). 
 \end{align*}

On the other hand, since $L^i$ are increasing and $    P^i_t \geq    P^{j}_t-c^{i,j}_t$, we obtain
 \begin{align*}
    d\big(   P_t^{E^{\star}_t}e^{\int_t^Tr(s)ds}\big) &=e^{\int_t^Tr(s)ds} \Big( - dL_t^{E^{\star}_{t-}}-H_t^{E^{\star}_t} (    Q_t^{E^{\star}_t},\beta_t)dt+    Q_t^{E^{\star}_t} dX_t+\big(   P^{E^{\star}_t }_t -    P_t^{E^{\star}_{t-}}\big)dN^{\star}_t\Big),\\
    &\leq e^{\int_t^Tr(s)ds} \Big(f_t^{E^{\star}_t} (\alpha^{\star}_t,\beta_t)dt+    Q_t^{E^{\star}_t} dW_t^{E^{\star},\alpha^{\star} , \beta} +c_t^{E^{\star}_{t-}, E^{\star}_t }dN^{\star}_t\Big).
 \end{align*}

Let $M^{\xi}_t:=\big(   P^{E^{\star}_t}_t - Y^{\xi;E^{\star}_t}_t\big)e^{\int_t^Tr(s)ds}$,   $t\in[0,T].$ Then we have $$
M^{\xi}_t \leq M^{\xi}_{0-}+ \int_0^te^{\int_s^Tr(s')ds'}\big(   Q_s^{E^{\star}_s}  - Z_s^{\xi;E^{\star}_s} \big)dW_s^{E^{\star},\alpha^{\star} , \beta}, 
$$
Since $   Q^i,Z^{\xi;i} \in \H$ for every $i\in \M$, we see by the BDG inequality that $M^{\xi}$ is a supermartingale. Since $g$ is non-decreasing, $g^{\conc}$ is also non-decreasing, combine with Jensen's inequality, we derive
\begin{align}
\E^{E^{\star},\alpha^{\star},\beta}\Big[g\big(M_T^{\xi}\big)\Big]\leq \E^{E^{\star},\alpha^{\star},\beta}\Big[g^{\conc}\big(M_T^{\xi}\big)\Big]&\leq g^{\conc}\Big(\E^{E^{\star},\alpha^{\star},\beta}\big[M_T^{\xi}\big]\Big)\\&\leq g^{\conc}(M^{\xi}_{0-}) =  g^{\conc}\Big(\big(   P^{\beta,0}_0-Y_0^{\xi,0}\big)e^{\int_0^Tr(s)ds}\Big)  .\label{eq:step1_thm_5.3}
 \end{align}

Since $(\xi,\beta)\in\C_{RS}$, by Proposition~\ref{prop:solve_agent_problem_regime}, we know $ Y_0^{\xi,0}\geq R$. Since $g $ is non-decreasing, we have 
$$
\E^{E^{\star},\alpha^{\star},\beta}\Big[g\big(M_T^{\xi}\big)\Big]\leq   g^{\conc}\Big(\big(   P^{\beta,0}_0 -  R\big)e^{\int_0^Tr(s)ds}\Big).
$$
Take supremum on $\xi$, we obtain $V^P_{\beta}\leq g^{\conc}\Big(\big(   P^{\beta,0}_0-R\big)e^{\int_0^Tr(s)ds}\Big)$.

\textit{Step 2.}
We will prove $V^P_{\beta}\geq g^{\conc}\Big(\big(   P^{\beta,0}_0-R\big)e^{\int_0^Tr(s)ds}\Big)$. Let $\bar{x}:=e^{\int_0^T r(s) d s}\big(   P_0^{ \beta,0}-R\big)$, and fix $\eta>0$. Since $g^{\conc}<\infty$, we can choose $x_1, x_2 \in \mathbb{R}$ and $\lambda \in(0,1)$ such that
$$
    \bar{x}=\lambda x_1+(1-\lambda) x_2,\qquad
    \lambda g\left(x_1\right)+(1-\lambda) g\left(x_2\right) \geq g^{\conc}(\bar{x})-\eta .
$$

Let $X^1$ be the first coordinate of $X$, and let $\Phi$ denote the probability distribution function of a standard 1-dimensional Gaussian variable $\mathcal{N}(0,1)$. Define
$$
A_{\varepsilon}:=\left\{X_{\varepsilon}^1-X_0^1 \leq \sqrt{\varepsilon} \Phi^{-1}(\lambda)\right\} .
$$
Let $ M:= \|\mu\|_{\infty}$, since for every $(E,\alpha,\beta)\in\Es\times \A\times\B$, $X^1_{\varepsilon} - X^1_0 - W^{E,\alpha,\beta,1}_{\varepsilon}= \int_0^{\varepsilon}\mu_s^{E_s,1}(\alpha_s,\beta_s)ds \in[ -M\varepsilon,M\varepsilon]$,
therefore, $
\big\{W_{\varepsilon}^{E,\alpha,\beta,1} \leq \sqrt{\varepsilon} \Phi^{-1}(\lambda)-M \varepsilon\big\} \subset A_{\varepsilon} \subset\big\{W_{\varepsilon}^{E,\alpha,\beta,1} \leq \sqrt{\varepsilon} \Phi^{-1}(\lambda)+M \varepsilon\big\},$ implying that $
\Phi\Big(\Phi^{-1}(\lambda)-M\sqrt{\varepsilon}\Big)\leq\P^{E,\alpha,\beta}\big(A_{\varepsilon} \big)\leq  \Phi\Big(\Phi^{-1}(\lambda)+M\sqrt{\varepsilon}\Big).
$
Consequently,
\begin{align}
    \label{eq:A_vepsilon_converges_to_lambda}
\lim_{\varepsilon\rightarrow 0^+ }\sup_{(E,\alpha,\beta)\in\Es\times \A\times\B}\Big|\mathbb{P}^{E, \alpha, \beta}\big(A_{\varepsilon}\big) -  \lambda \Big| = 0.
\end{align}
Define the contract
$$
\xi^{\varepsilon}:=\ell-Q_{\varepsilon}+\delta_{\varepsilon},
~~\text{where}~~
Q_{\varepsilon}:=x_1 \mathds{1}_{A_{\varepsilon}}+x_2 \mathds{1}_{A_{\varepsilon}^c},
~~\text{and}~~
\delta_{\varepsilon}:=\sup _{E, \alpha}\left|\mathbb{E}^{E, \alpha, \beta}\left[Q_{\varepsilon}\right]-\bar{x}\right|,
$$
and notice that $\lim_{\varepsilon\rightarrow 0^+}\delta_{\varepsilon} = 0$, by \eqref{eq:A_vepsilon_converges_to_lambda}. As $Q_{\varepsilon}$ is bounded, $\xi^{\varepsilon}\in\L$. Then for every $(E, \alpha)\in\Es\times\A$,
\begin{align*}
J_{R S}^A\left(\xi^{\varepsilon}, \beta ; E, \alpha\right)&=J_{R S}^A(\ell, \beta ; E, \alpha)-e^{-\int_0^Tr(s)ds}\Big( \mathbb{E}^{E, \alpha, \beta}[Q_{\varepsilon}]-\delta_{\varepsilon} \Big)
\geq J_{R S}^A(\ell, \beta ; E, \alpha)-e^{-\int_0^Tr(s)ds}\bar{x}.
\end{align*}

Taking the supremum over $(E, \alpha)$, we derive:

$$
V_{R S}^A\left(\xi^{\varepsilon}, \beta\right) \geq V_{R S}^A(\ell, \beta)-e^{-\int_0^Tr(s)ds}\bar{x} =     P_0^{\beta,0} - e^{-\int_0^Tr(s)ds}\bar{x}  =R,
$$
which implies that $(\xi^{\varepsilon},\beta)\in \C_{RS}$, i.e. $\xi^{\varepsilon}$ is an admissible contract. Now let $\left(E^{\varepsilon}, \alpha^{\varepsilon}\right)$ be any  optimal response to $\left(\xi^{\varepsilon}, \beta\right)$, and denote $\E^\varepsilon:=\E^{E^{\varepsilon}\!, \alpha^{\varepsilon}\!,\beta}$ and $\P^\varepsilon$ the corresponding probability measure, then:
\begin{align*}
V^P_{\beta}
\geq
\E^\varepsilon\Big[g\left(\ell-\xi^{\varepsilon}\right)\Big] &=\E^\varepsilon\Big[g\left(Q_{\varepsilon}\!-\!\delta_{\varepsilon}\right)\Big] 
 = g(x_1 \!-\! \delta_{\varepsilon})\P^\varepsilon(A_{\varepsilon}) + g(x_2 \!-\! \delta_{\varepsilon})\P^\varepsilon(A_{\varepsilon}^c).
\end{align*}
 Since $g$ is continuous, we pass $\varepsilon$ to $0$ and derive $V^P_{\beta} \geq \lambda g\left(x_1\right)+(1-\lambda) g\left(x_2\right) \geq g^{\text {conc }}(\bar{x})-\eta\longrightarrow g^{\text {conc }}(\bar{x})$ as $\eta\searrow 0$.

\medskip
\textit{Step 3.} 
Combine the previous two steps we conclude that $V^P_{\beta}  =g^{\text {conc }}(\bar{x})$. Since $g^{\conc}$ is non-decreasing, we see that\begin{align*}
    V^P_{RS} = \sup_{\beta\in \B}g^{\conc}\Big(\big(   P^{\beta,0}_0-R\big)e^{\int_0^Tr(s)ds}\Big) =g^{\conc}\bigg(\Big( \sup_{\beta\in \B}   P^{\beta,0}_0-R\Big)e^{\int_0^Tr(s)ds}\bigg).
\end{align*}

Now the optimization problem $\sup_{\beta\in \B}   P_0^{\beta,0}$ is a standard regime switching problem, which can be solved by a simple verification argument similar to Proposition~\ref{prop:solve_agent_problem_regime}. This  yields  $\sup_{\beta\in\B}   P_0^{\beta,0} = \bar{   P}^0_0,$
and we finally conclude $
V^{P}_{RS} =  g^{\conc}\Big(\big(\bar{   P}_0^{0} - R\big)e^{\int_0^Tr(s)ds}\Big).
$
When $\big(\bar{   P}_0^{0} - R\big)e^{\int_0^Tr(s)ds}\in \{g = g^{\conc}\}$, it is easy to see that the maximum is attained at $(\xi^{\star},\beta^{\star})$.
\end{proof}

\begin{remark}
    In the proof above, we use the boundedness of $\mu$ to derive Equation~\eqref{eq:A_vepsilon_converges_to_lambda}. The same conclusion remains valid under the weaker assumption that, $$
\lim_{\varepsilon\rightarrow0^+}\varepsilon^{\frac{p}{2}-1}\sup_{(E,\alpha,\beta)\in\Es\times \A\times\B}\E^{E,\alpha,\beta}\mbrac[4]{\int_0^{\varepsilon}\sup_{i\in\M}\abs{\mu^i_t}^pdt }= 0 ,
    $$
 for some $p>1$.    Therefore, combine with Remark~\ref{rmk:relax_boundedness_mu}, this allows us to
    consider a broader class of $\mu$, which need not be bounded.
\end{remark}

\section{Conclusion}
\label{sec:conclusion}

We develop a BSDE approach to finite horizon Principal--Agent problems, which characterizes the principal's value and constructs an optimal contract under suitable conditions. We study a risk-neutral agent in a path-dependent setting with a possibly non-concave utility for the principal, and we also consider exponential utilities as an example with a risk-averse agent. We then turn to a model with unobservable regime switching. Here, the usual continuation value reduction does not lead to a standard stochastic control problem, but the principal's value still admits a BSDE characterization. This example suggests that the method may be useful in other contracting models where the usual control formulation is not available.

\appendix
\section{Computational details}\label{app:compute}

\begin{proposition}\label{prop:fully_nonlinear_to_semilinear}
Let $g$ be non-decreasing with $g^{\conc}\in C^{2}(\R)$, $r$ be continuous, and $w\in C^{1,2}([0,T)\times \R^d)\cap C([0,T]\times \R^d)$ a classical solution to Equation \eqref{eq:semi_linear_PDE_of_w}. Then the value function $v(t,x,y):= g^{\conc}\big((w(t,x) - y)e^{\int_t^Tr(s)ds}\big)$ is a classical solution of \eqref{eq:HJB_fully_non_linear_non_concav}. 
\end{proposition}
\begin{proof}
    Without loss of generality, we assume $g$ is concave. The terminal condition is obvious. Let $\bar{\mathcal{R}}_t := e^{\int_t^Tr(s)ds}$ Denote $g'\big((w(t,x) - y)\bar{\mathcal{R}}_t \big)\bar{\mathcal{R}}_t $ by $g'_t$ and $g''\big((w(t,x) - y)\bar{\mathcal{R}}_t \big)\bar{\mathcal{R}}_t ^2$ by $g''_t$. Then $$
        \partial_t v = g'_t(\partial_t w-r(w-y)), \quad \partial_y v = -g'_t, \quad \partial_{x} v = g'_t \partial_{x} w, $$
        $$
       \partial_{yy} v = g''_t, \quad \partial^2_{x y} v = - g''_t\partial_{x} w, \quad \Delta_{xx}v = g'' _t|\partial_{x} w|^2 + g' _t\Delta_{xx}w. 
$$
Let $\mu^{\star}(t,x,z,b):=\mu(t,x,\alpha^{\star}(t,x,z,b),b)\in\partial_z H(t,x,z,b) $, by definition,
\begin{align*}
F(t,x,Dv,D^2v)&= \sup_{\substack{z\in\R^d\\b\in B}}\Big[
\mu^{\star}(t,x,z,b)\cdot\big(\partial_{x}v(t,x,y)+z\,\partial_yv(t,x,y)\big)
- H(t,x,z,b)\,\partial_yv(t,x,y) \\
&\hspace{4.5em}
+\tfrac12 |z|^2\,\partial_{yy}v(t,x,y)
+ z\cdot\partial^2_{x y}v(t,x,y)\Big]
\\ &
= \sup_{\substack{z\in\R^d\\b\in B}}\Big[g'_t(\mu^{\star}(t,x,z,b)\cdot (\partial_{x} w-z)+H(t,x,z,b)) +g''_t(\tfrac{1}{2}|z|^2-\partial_{x}w \cdot z)\Big].
\end{align*}
Since $H$ is convex in $z$, and $\mu^{\star}(t,x, z,b)\in \partial_zH(t,x, z,b)$, so we obtain: for every $(t,x,b)$, $z\mapsto\mu^{\star}(t,x, z,b)\!\cdot\! (\partial_{x} w-z)+H(t,x,z,b)$ attains its maximum at $z = \partial_{x} w$. On the other hand, $\tfrac{1}{2}|z|^2-\partial_{x}w \cdot z$  attains its minimum at $z = \partial_{x} w$ as well. By assumption, $g'_t\geq 0$, $g''_t\leq 0$, so $$
 F(t,x,Dv,D^2v)= \sup_{b\in B} \Big[g'_tH( t,x,\partial_{x} w,b) -\tfrac{1}{2}g''_t|\partial_{x}w|^2\Big]=g'_t\bar{H}( t,x,\partial_x w) -\tfrac{1}{2}g''_t|\partial_{x}w|^2,
$$
and then
$-\partial_t v
-\tfrac12 \Delta_{xx}v-r\,y\,\partial_yv-F(\cdot,Dv,D^2v)= g'_t \big(-\partial_t w -\tfrac{1}{2}\Delta_{xx}w +rw -\bar{H}( \cdot,\partial_{x} w) \big)= 0.$
\end{proof}

\section{Verification of the Assumption for Quadratic BSDE}\label{app:quadratic_bsde}
In this section we verify the $F$ defined by Equation~\eqref{eq:def_F} satisfies the standard assumption for generators of quadratic BSDE. First as all $p\in \partial_z H_t(z,b)$ satisfies $|p|\leq \|\mu\|_{\infty}$, and $|H_t(z,b)|\leq C_H(1+|z|)$ for some constant $C_H>0$, it follows that $\|F_{\cdot}(0)\|_{\infty}<\infty$. It remains to verify that
$$
|F_t(q) - F_t(q')|\leq L_F(1+|q|+|q'|)|q-q'|.
$$
It is easy to check that there exist constants $C_1$, $C_2$ depending only on $\eta_A, \eta_P, \|\mu\|_{\infty},\|f\|_{\infty}$, s.t. \begin{equation}
    \label{eq:app_quadra_property1}
\tilde{F}_t(q, z, \zeta, b)\leq C_1 - C_2\big(|z|^2+|q-z|^2\big).
\end{equation}
On the other hand, pick $b_0\in B$ and ${\zeta}_0\in\partial_z H_t(0,b_0)$   \begin{equation} \label{eq:app_quadra_property2}
F_t(q)\geq \tilde{F}_t(q,0,{\zeta}_0,b_0) = q\cdot {\zeta}_0+H_t(0,b_0)-\tfrac{\eta_P}{2}|q|^2\geq -C(1+|q|^2),   
\end{equation}
for some constant $C>0$ depending only on $\eta_P, \|\mu\|_{\infty},\|f\|_{\infty}$.
Now fixed $q\in\R^d$, for every $\varepsilon\in(0,1)$, take $(z_{\varepsilon},b_{\varepsilon},{\zeta}_{\varepsilon})$ that satisfies \begin{equation}
    \label{eq:p_varepsilon_satisf}
{\zeta}_{\varepsilon}\in \partial_z H_t(z_{\varepsilon},b_{\varepsilon}),\qquad F_t(q)\leq \tilde{F}_t(q,z_{\varepsilon},{\zeta}_{\varepsilon},b_{\varepsilon})+\varepsilon.
\end{equation}
By Equation~\eqref{eq:app_quadra_property1}-\eqref{eq:app_quadra_property2}
\begin{align*}
    C_1 - C_2\big(|z_{\varepsilon}|^2+|q-z_{\varepsilon}|^2\big)\geq \tilde{F}_t(q, z_{\varepsilon}, {\zeta}_{\varepsilon}, b_{\varepsilon})\geq F_t(q)-\varepsilon\geq -C(1+|q|^2)-1.
\end{align*}
So $
C(1+|q|^2)\geq |q-z_{\varepsilon}|^2,
$
for some $C>0$ depending only on $\eta_A,\eta_P, \|\mu\|_{\infty},\|f\|_{\infty}$.
For every $q,q'\in\R^d$, $\varepsilon\in(0,1)$ and $(z_{\varepsilon},b_{\varepsilon},{\zeta}_{\varepsilon})$ that satisfies Equation~\eqref{eq:p_varepsilon_satisf} 
\begin{align*}
    F_t(q)-F_t(q') &\leq \tilde{F}_t(q, z_{\varepsilon}, {\zeta}_{\varepsilon}, b_{\varepsilon})-\tilde{F}_t(q', z_{\varepsilon}, {\zeta}_{\varepsilon}, b_{\varepsilon})+\varepsilon \\
    &=(q-q') \cdot {\zeta}_{\varepsilon}-\tfrac{\eta_P}{2}\big(|q-z_{\varepsilon}|^2-|q'-z_{\varepsilon}|^2\big)+\varepsilon\\
&= (q-q')\cdot\big( {\zeta}_{\varepsilon}-\eta_P(q-z_{\varepsilon}) \big)+\tfrac{\eta_P}{2}|q - q'|^2+\varepsilon\\
&\leq  |q - q'| \|\mu\|_{\infty}+\eta_P|q - q'| \cdot|q-z_{\varepsilon}|+\tfrac{\eta_P}{2}|q - q'|^2+\varepsilon\\
&\leq  |q - q'| \|\mu\|_{\infty}+\sqrt{C}\eta_P|q - q'|(1+|q|)+\tfrac{\eta_P}{2}|q - q'|\cdot (|q| + |q'|)+\varepsilon\\
&=  |q - q'| \big(\|\mu\|_{\infty}+\sqrt{C}\eta_P+(\sqrt{C}+\tfrac{1}{2})\eta_P|q|+\tfrac{\eta_P}{2} |q'|\big)+\varepsilon.
\end{align*}
Let $\varepsilon\rightarrow 0$ we have $F_t(q) - F_t(q')\leq L_F(1+|q|+|q'|)|q-q'|$, where $L_F := \|\mu\|_{\infty}+2\sqrt{C}\eta_P+\eta_P$. By symmetry we complete the proof.
\bibliographystyle{plainnat_surname_initials}
\bibliography{bibl.bib}

\end{document}